\documentclass[10pt,oneside,reqno,twoside]{amsart}
\usepackage[foot]{amsaddr} %% affiliation foot first page
\usepackage[utf8]{inputenc}
\usepackage{amsmath}
\usepackage{amsfonts}
\usepackage{amssymb}
\usepackage{setspace}
\usepackage{amssymb}
\usepackage[shortlabels]{enumitem}
\usepackage{xcolor}
\usepackage{graphicx}
\usepackage{appendix}
\usepackage{amsthm}
\newtheorem{thm}{Theorem}[section]

\numberwithin{equation}{section}
\usepackage{bbm}
\newtheorem{rmk}{Remark}[section]

\newtheorem{prop}{Proposition}[section]

\newtheorem{cor}{Corollary}[section]
\usepackage{geometry}
\usepackage{fancyhdr}
\usepackage{hyperref}
\hypersetup{
	colorlinks=true,
	linkcolor=blue,
	filecolor=magenta,      
	urlcolor=blue,
	pdftitle={Overleaf Example},
	citecolor=red,
	pdfpagemode=FullScreen,
}
\usepackage{mfirstuc}
\makeatletter
\def\@setauthors{%
	\begingroup
	\def\thanks{\protect\thanks@warning}%
	\trivlist
	\centering\footnotesize \@topsep30\p@\relax
	\advance\@topsep by -\baselineskip
	\item\relax
	\author@andify\authors
	\def\\{\protect\linebreak}%
	\authors%
	\ifx\@empty\contribs
	\else
	,\penalty-3 \space \@setcontribs
	\@closetoccontribs
	\fi
	\endtrivlist
	\endgroup
}
\date{}
\begin{document}
	
	\title[Inverse initial problem, Nash strategy, Stochastic parabolic equations]{Inverse initial problem under Nash strategy for stochastic reaction-diffusion equations with dynamic boundary conditions}
	
	%\makeatletter
	%\@namedef{subjclassname@2020}{%
		%  \textup{2020} Mathematics Subject %Classification}
	%\makeatother
	%\subjclass[2020]{Primary: 35R12, 49N25; Secondary: 93C27.}
	
	\author[A. Elgrou]{\large Abdellatif Elgrou$^1$}
	\address{$^1$Cadi Ayyad University, Faculty of Sciences Semlalia, LMDP, UMMISCO (IRD-UPMC), P.B. 2390, Marrakesh, Morocco.}
	\email{abdoelgrou@gmail.com}

	\author[L. Maniar]{\large Lahcen Maniar$^{1,2}$}
	\address{$^2$University Mohammed VI  Polytechnic, Vanguard Center, Benguerir, Morocco.}
	\email{maniar@uca.ma}
	\email{Lahcen.Maniar@um6p.ma}

	\author[O. Oukdach]{\large Omar Oukdach$^3$}
	\address{$^3$Moulay Ismaïl University of Meknes, FST Errachidia,  MSISI Laboratory, AM2CSI Group, BP 50, Boutalamine, Errachidia, Morocco.}
	\email{omar.oukdach@gmail.com}

	\begin{abstract}
		In this paper, we study a multi-objective inverse initial problem with a Nash strategy constraint for forward stochastic reaction-diffusion equations with dynamic boundary conditions, where both the volume and surface equations are influenced by randomness. The objective is twofold: first, we maintain the state close to prescribed targets in fixed regions using two controls; second, we determine the history of the solution from observations at the final time. To achieve this, we establish new Carleman estimates for forward and backward equations, which are used to prove an interpolation inequality for a coupled forward-backward stochastic system. Consequently, we obtain two results: backward uniqueness and a conditional stability estimate for the initial conditions.
	\end{abstract}
	\maketitle

	\maketitle

	\section{Introduction}
	Roughly speaking, an inverse problem for partial differential equations (PDEs, for short) is the problem of determining some of the coefficients (such as potentials, initial conditions, boundary conditions, or source terms) involved in the equation based on observations of the solution. Typically, an inverse problem arises when a quantity cannot be measured directly but can only be inferred from the observation of its effects. 
	
	Compared to the extensive literature devoted to the study of inverse problems for deterministic PDEs, stochastic equations are much less studied. Numerous methods and techniques have been developed to address these problems. Among these, Carleman estimates are some of the most useful and efficient tools. For some results on different inverse problems, we refer to \cite{Dinvpb2, DIP1, Dinvpb1, DIP2, DIP3, Yamam2009invePrb} for the deterministic setting and to \cite{StIP3, StIP2} for some stochastic equations, along with the survey article \cite{StIP1} for a detailed presentation on the subject.
	
	Among important inverse problems, the inverse initial value problem is of particular interest. For example, as explained in \cite{DIP4}, it is practically useful for determining past temperatures based on the current temperature distribution, which is relevant to fields such as thermo-archaeology and developing policies for global warming by estimating past temperatures. See, for instance, \cite{DIP1, DIP4} and references therein.
	
	Unlike classical inverse problems, the present study addresses a multi-objective inverse problem for stochastic parabolic equations with dynamic boundary conditions. More precisely, we combine some techniques  of classical inverse problems with concepts from game theory. The resolution process can be summarized in   two steps. The first step, achieved by applying two localized controls, consists of maintaining the state close to prescribed targets in fixed regions. The second step aims  to determine the history of the solution process from observations made at the final time. As we will demonstrate, after the first step, the multi-objective inverse problem is reduced to a classical inverse problem, which will be addressed by proving an interpolation inequality for a coupled forward-backward stochastic system. The proofs rely on some new Carleman estimates for a class of forward and backward stochastic parabolic equations with dynamic boundary conditions. To the authors' best knowledge, the present paper is the first to address inverse problems in the context of stochastic parabolic equations with dynamic boundary conditions. We note that this type of boundary conditions has been considered in several papers. We refer to \cite{ACMO20, Dinvpb2, Dinvpb1} for some results concerning inverse problems, and \cite{OuBoMa21, BCMO20, BoMaOuNash, maniar2017null} for some control issues in the deterministic case. For stochastic equations with dynamic boundary conditions, see \cite{elgrouDBC, BackSPEwithDBC}. We also refer to the works \cite{BEMOstoch, stackNashnullcont24}, where we have considered some multi-objective optimal control problems for certain stochastic parabolic equations. 
	
	Multi-objective inverse problems have many practical applications. For example, if 
	$y$ represents the concentration of chemicals, we aim to keep both the concentration 
	$y$ and its rate of change close to their desired values by adding water to   the chemical or evaporating it. At the same time, the initial concentration is expected to be determined from observations made at the terminal time. Despite their importance, this type of problem is rarely studied in the literature. We refer to \cite{Milti-IP4}, where the authors considered degenerate stochastic parabolic equations. As far as we know, this cited article is the only one to deal with multi-objective inverse problems.\\
	
	The plan for the rest of the paper is as follows: In the next section, we formulate the problem under consideration and present the main results of this paper. In Section \ref{secc3}, we characterize the Nash equilibrium. In Section \ref{sec3}, we establish two necessary Carleman estimates for forward and backward stochastic parabolic equations with dynamic boundary conditions, respectively. Section \ref{sec4} is devoted to proving our main results.
	\section{Problem Formulation and the Main Results}
	Let us introduce some necessary notations. Let \( T > 0 \) and \( G \subset \mathbb{R}^N \) (with \( N \geq 2 \)) be a nonempty open bounded domain with a smooth boundary \( \Gamma = \partial G \). Let the sets \( G_1 \), \( G_2 \), \( G_{1,d} \), and \( G_{2,d} \) be any nonempty open subsets of \( G \). We also denote by \( \chi_{\mathcal{O}} \) the characteristic function of a subset \( \mathcal{O} \subset G \). Define
	\[
	Q = (0, T) \times G \quad \textnormal{and} \quad \Sigma = (0, T) \times \Gamma.
	\]
	
	Let \((\Omega, \mathcal{F}, \{\mathcal{F}_t\}_{t \in [0,T]}, \mathbb{P})\) be a fixed complete filtered probability space on which a one-dimensional standard Brownian motion \(W(\cdot)\) is defined, such that \(\{\mathcal{F}_t\}_{t \in [0,T]}\) is the natural filtration generated by \(W(\cdot)\) and augmented by all the \(\mathbb{P}\)-null sets in \(\mathcal{F}\). For a Banach space \(\mathcal{X}\), we denote by \(C([0,T]; \mathcal{X})\) the Banach space of all \(\mathcal{X}\)-valued continuous functions defined on \([0,T]\); and \(L^2_{\mathcal{F}_t}(\Omega; \mathcal{X})\) denotes the Banach space of all \(\mathcal{X}\)-valued \(\mathcal{F}_t\)-measurable random variables \(X\) such that \(\mathbb{E}\big(\vert X \vert_\mathcal{X}^2\big) < \infty\), with the canonical norm. Additionally, \(L^2_\mathcal{F}(0,T; \mathcal{X})\) indicates the Banach space consisting of all \(\mathcal{X}\)-valued \(\{\mathcal{F}_t\}_{t \in [0,T]}\)-adapted processes \(X(\cdot)\) such that \(\mathbb{E}\big(\vert X(\cdot) \vert^2_{L^2(0,T; \mathcal{X})}\big) < \infty\), with the canonical norm; and \(L^\infty_\mathcal{F}(0,T; \mathcal{X})\) is the Banach space consisting of all \(\mathcal{X}\)-valued \(\{\mathcal{F}_t\}_{t \in [0,T]}\)-adapted bounded processes, with the essential supremum norm denoted by \(|\cdot|_\infty\). Finally, \(L^2_\mathcal{F}(\Omega; C([0,T]; \mathcal{X}))\) defines the Banach space consisting of all \(\mathcal{X}\)-valued \(\{\mathcal{F}_t\}_{t \in [0,T]}\)-adapted continuous processes \(X(\cdot)\) such that \(\mathbb{E}\big(\max_{0\leq t\leq T}\vert X(t) \vert^2_{\mathcal{X}}\big) < \infty\), with the canonical norm. We now introduce the following space
	\[
	\mathbb{L}^2 := L^2(G, \, dx) \times L^2(\Gamma, \, d\sigma),
	\]
	where \(dx\) denotes the Lebesgue measure in \(G\) and \(d\sigma\) indicates the surface measure on \(\Gamma\). Equipped with the following standard inner product
	\[
	\langle (y, y_\Gamma), (z, z_\Gamma) \rangle_{\mathbb{L}^2} = \langle y, z \rangle_{L^2(G)} + \langle y_\Gamma, z_\Gamma \rangle_{L^2(\Gamma)},
	\]
	the space \(\mathbb{L}^2\) is a Hilbert space. Moreover, we adopt the following notations:
	\[
	\mathcal{V}_i = L^2_\mathcal{F}(0,T; L^2(G_i)), \quad \mathcal{V}_{i,d} = L^2_\mathcal{F}(0,T; L^2(G_{i,d})) \qquad \textnormal{for} \;\; i=1,2.
	\]
	Consider the following stochastic reaction-diffusion system
	\begin{equation}\label{eqq1.1}
		\begin{cases}
			\begin{array}{ll}
				dy - \Delta y \, dt = [a_1 y + \chi_{G_1}(x) v_1 + \chi_{G_2}(x) v_2] \, dt + a_2 y \, dW(t) & \textnormal{in } Q,\\
				dy_\Gamma - \Delta_\Gamma y_\Gamma \, dt + \partial_\nu y \, dt = b_1 y_\Gamma \, dt + b_2 y_\Gamma \, dW(t) & \textnormal{on } \Sigma,\\
				y_\Gamma = y \vert_\Gamma & \textnormal{on } \Sigma,\\
				(y, y_\Gamma) \vert_{t=0} = (y_0, y_{\Gamma,0}) & \textnormal{in } G \times \Gamma,
			\end{array}
		\end{cases}
	\end{equation}
	where \((y_0, y_{\Gamma,0}) \in L^2_{\mathcal{F}_0}(\Omega; \mathbb{L}^2)\) is the initial state, \((y, y_\Gamma)\) is the state variable, and \((v_1, v_2) \in \mathcal{V}_1 \times \mathcal{V}_2\) is a pair of controls. The coefficients \(a_1\), \(a_2\), \(b_1\) and \(b_2\)  are assumed to be
	\[
	a_1 \in L_\mathcal{F}^\infty(0,T; L^\infty(G)), \quad a_2 \in L_\mathcal{F}^\infty(0,T; W^{1,\infty}(G)),
	\]
	\[
	b_1 \in L_\mathcal{F}^\infty(0,T; L^\infty(\Gamma)), \quad b_2 \in L_\mathcal{F}^\infty(0,T; W^{1,\infty}(\Gamma)).
	\] 
	Throughout this paper, \(C\) denotes a positive constant depending only on \(G\), \(G_{1}\), \(G_{2}\), \(T\), \(a_1\), \(a_2\), \(b_1\), and \(b_2\), which may change from one place to another. In equation \eqref{eqq1.1}, \(y \vert_\Gamma\) denotes the trace of the function \(y\), and \(\partial_\nu y = (\nabla y \cdot \nu) \vert_{\Sigma}\) is the normal derivative of \(y\), where \(\nu\) is the outer unit normal vector at the boundary \(\Gamma\). This normal derivative plays the role of the coupling term between the volume and surface equations. For more details and the physical interpretation of dynamical boundary conditions, also called Wentzell boundary conditions, we refer to \cite{cocang2008, Gal15, Gol06}. System \eqref{eqq1.1} (with \(v_1 \equiv v_2 \equiv 0\)) describes various diffusion phenomena, such as thermal processes. These systems, subject to stochastic disturbances, also account for small independent changes during the heat process. Here, we study systems with stochastic dynamic boundary conditions that take into account the dynamic interaction between the domain and the boundary. For further details on the physical model described by such systems, see, for instance, \cite{Pbrrune, chuesBj}.

	Now, let us recall some useful differential operators defined on \(\Gamma\). Firstly, the tangential gradient of a function \(y_\Gamma: \Gamma \rightarrow \mathbb{R}\) is defined by
	\[
	\nabla_{\Gamma} y_\Gamma = \nabla y - \partial_{\nu} y \, \nu,
	\]
	where \(y\) is an extension of \(y_\Gamma\) up to an open neighborhood of \(\Gamma\). Secondly, the tangential divergence of a function \(y_\Gamma: \Gamma \rightarrow \mathbb{R}^N\), such that \(y = y_\Gamma\) on \(\Gamma\), is defined as
	\[
	\operatorname{div}_{\Gamma} (y_\Gamma) = \operatorname{div}(y) - \nabla y \,\nu  \cdot \nu,
	\]
	where \(\nabla y\) is the standard Jacobian matrix of \(y\).
	
	Then, we define the Laplace-Beltrami operator of a function \(y_\Gamma: \Gamma \rightarrow \mathbb{R}\) as follows:
	\[
	\Delta_{\Gamma} y_\Gamma = \operatorname{div}_{\Gamma} (\nabla_{\Gamma} y_\Gamma).
	\]
	
	As in the case of the standard Sobolev spaces \(H^1(G)\) and \(H^2(G)\),  Sobolev spaces   \(H^1(\Gamma)\) and \(H^2(\Gamma)\) are defined as follows:
	\[
	H^1(\Gamma) = \left\{ y_\Gamma \in L^2(\Gamma) \mid \nabla_\Gamma y_\Gamma \in L^2(\Gamma; \mathbb{R}^N) \right\},
	\]
	and
	\[
	H^2(\Gamma) = \left\{ y_\Gamma \in L^2(\Gamma) \mid \nabla_\Gamma y_\Gamma \in H^1(\Gamma; \mathbb{R}^N) \right\},
	\]
	endowed with the following norms respectively:
	\[
	|y_\Gamma|_{H^1(\Gamma)} = \langle y_\Gamma, y_\Gamma \rangle_{H^1(\Gamma)}^{\frac{1}{2}}, \quad \textnormal{with} \quad \langle y_\Gamma, z_\Gamma \rangle_{H^1(\Gamma)} = \int_{\Gamma} y_\Gamma z_\Gamma \, d\sigma + \int_{\Gamma} \nabla_{\Gamma} y_\Gamma \cdot \nabla_{\Gamma} z_\Gamma \, d\sigma,
	\]
	and
	\[
	|y_\Gamma|_{H^2(\Gamma)} = \langle y_\Gamma, y_\Gamma \rangle_{H^2(\Gamma)}^{\frac{1}{2}}, \quad \textnormal{with} \quad \langle y_\Gamma, z_\Gamma \rangle_{H^2(\Gamma)} = \int_{\Gamma} y_\Gamma z_\Gamma \, d\sigma + \int_{\Gamma} \Delta_{\Gamma} y_\Gamma \Delta_{\Gamma} z_\Gamma \, d\sigma.
	\]
	In the subsequent sections, for the Laplace-Beltrami operator \(\Delta_\Gamma\), we commonly use the following surface divergence formula
	\[
	\int_\Gamma \Delta_\Gamma y_\Gamma\, z_\Gamma \, d\sigma = -\int_\Gamma \nabla_\Gamma y_\Gamma \cdot \nabla_\Gamma z_\Gamma \, d\sigma, \qquad y_\Gamma \in H^2(\Gamma), \; z_\Gamma \in H^1(\Gamma).
	\]
	We also define the following Hilbert spaces
	\[
	\mathbb{H}^k = \left\{ (y, y_\Gamma) \in H^k(G) \times H^k(\Gamma) \mid y_\Gamma = y \vert_\Gamma \right\}, \quad k = 1,2,
	\]
	which can be viewed as a subspace of \(H^k(G) \times H^k(\Gamma)\) with the natural topology inherited from \(H^k(G) \times H^k(\Gamma)\).
	
	From \cite{elgrouDBC}, the equation \eqref{eqq1.1} is well-posed, i.e., for any \((y_0, y_{\Gamma,0}) \in L^2_{\mathcal{F}_0}(\Omega; \mathbb{L}^2)\) and \((v_1, v_2) \in \mathcal{V}_1 \times \mathcal{V}_2\), \eqref{eqq1.1} admits a unique mild solution
	\[
	(y, y_\Gamma) \in L^2_\mathcal{F}(\Omega; C([0, T]; \mathbb{L}^2)) \cap L^2_\mathcal{F}(0, T; \mathbb{H}^1).
	\]
	Moreover, there exists a constant \(C > 0\) such that
	\begin{align}\label{forwenergyes}
		\begin{aligned}
			& \, \vert (y, y_\Gamma) \vert_{L^2_\mathcal{F}(\Omega; C([0, T]; \mathbb{L}^2))} + \vert (y, y_\Gamma) \vert_{L^2_\mathcal{F}(0, T; \mathbb{H}^1)} \\
			& \leq C \, \left( |(y_0, y_{\Gamma,0})|_{L^2_{\mathcal{F}_0}(\Omega; \mathbb{L}^2)} + |v_1|_{L^2_\mathcal{F}(0, T; L^2(G_1))} + |v_2|_{L^2_\mathcal{F}(0, T; L^2(G_2))} \right).
		\end{aligned}
	\end{align}
	
	We now provide some remarks.
	\begin{rmk}
		\begin{enumerate}[1.]
			\item In system \eqref{eqq1.1}, the coefficients \(a_1\), \(a_2\), \(b_1\), and \(b_2\) represent the reaction terms, while the diffusion is governed by \(\Delta\) in the volume equation and by \(\Delta_\Gamma\) in the surface equation.
			\item In this work, it is also possible to consider the following more general diffusion terms instead of \(\Delta\) (resp., \,\(\Delta_\Gamma\)):
			\[
			\textnormal{div}(A\nabla\cdot) \;\;\textnormal{(resp., } \,\textnormal{div}_\Gamma(A_\Gamma\nabla_\Gamma\cdot)),
			\]
			where \(A \in L^2_\mathcal{F}(\Omega;C^1([0,T];W^{2,\infty}(\overline{G};\mathbb{R}^{N\times N})))\) and \(A_\Gamma \in L^2_\mathcal{F}(\Omega;C^1([0,T];W^{2,\infty}(\Gamma;\mathbb{R}^{N\times N})))\) are symmetric and uniformly elliptic, i.e., there exists a constant \(a_0 > 0\) such that
			\begin{align*}
				& A(\omega,t,x)\xi\cdot\xi \geq a_0 |\xi|^2 \quad \textnormal{for all}\;\;(\omega,t,x,\xi)\in\Omega\times \overline{Q}\times\mathbb{R}^N, \\
				& A_\Gamma(\omega,t,x)\xi\cdot\xi \geq a_0 |\xi|^2 \quad \textnormal{for all}\;\;(\omega,t,x,\xi)\in\Omega\times \Sigma\times\mathbb{R}^N.
			\end{align*}
			See \cite{StIP2} for the case of Dirichlet boundary conditions.
			\item Due to the weighted identity method employed in the proof of the Carleman estimate for forward stochastic parabolic equations in Section \ref{sec3}, we require in equation \eqref{eqq1.1} that the coefficients \(a_2\) and \(b_2\) are in \(W^{1,\infty}\) with respect to the space variable. Thus, it is of interest to explore whether we can relax this assumption to \(L^\infty\) in space.
			\item In this study, it would be interesting to incorporate gradient terms \(\nabla y\) in the volume equation and \(\nabla_\Gamma y_\Gamma\) in the surface equation of the system \eqref{eqq1.1}. This extension will allow us to investigate a stochastic reaction-convection-diffusion system with dynamic boundary conditions, which will enhance the understanding of more complex physical phenomena.
		\end{enumerate}
	\end{rmk}
	
	The main purpose of this paper is to study a multi-objective inverse initial problem for equation \eqref{eqq1.1}. To formulate the problem under consideration, for two fixed target functions \(y_{i,d} \in \mathcal{V}_{i,d}\) (\(i=1,2\)), we introduce the following functionals
	\begin{equation}\label{Ji.12}
		J_i(v_1,v_2) = \frac{\alpha_i}{2} \mathbb{E} \int_0^T\int_{G_{i,d}} |y - y_{i,d}|^2 \, dx \, dt + \frac{\beta_i}{2} \mathbb{E} \int_0^T\int_{G_{i}} |v_i|^2 \, dx \, dt, \quad i=1,2,
	\end{equation}
	where \(\alpha_i\), \(\beta_i\) are positive constants and \((y,y_{\Gamma})\) is the solution of \eqref{eqq1.1} associated with the initial condition \((y_0, y_{\Gamma,0}) \in L^2_{\mathcal{F}_0}(\Omega; \mathbb{L}^2)\) and controls \((v_1,v_2) \in \mathcal{V}_1 \times \mathcal{V}_2\).
	
	We first find a Nash equilibrium \((v^{\star}_1, v^{\star}_2) \in \mathcal{V}_1 \times \mathcal{V}_2\) for \(J_1\) and \(J_2\), which is defined in the following sense: \((v^{\star}_1, v^{\star}_2) \in \mathcal{V}_1 \times \mathcal{V}_2\) is a Nash equilibrium for \(J_1\) and \(J_2\) if and only if 
	\begin{equation}\label{nshequi}
		\left\{
		\begin{array}{ll}
			J_1(v^{\star}_1, v^{\star}_2) = \min\limits_{v \in  \mathcal{V}_1} J_1(v, v^{\star}_2), \\
			J_2(v^{\star}_1, v^{\star}_2) = \min\limits_{v \in  \mathcal{V}_2} J_2(v^{\star}_1, v).
		\end{array}
		\right.
	\end{equation}
	Such property \eqref{nshequi} means that \((v^{\star}_1, v^{\star}_2)\) is a Nash equilibrium for \(J_1\) and \(J_2\) if and only if the strategy chosen by any control is optimal when all other control strategies are determined. Hence, fixing the obtained controls \((v^{\star}_1, v^{\star}_2)\) in \eqref{eqq1.1}, the multi-objective inverse initial problem is reduced to a classical inverse initial problem where the state of \eqref{eqq1.1} can be determined quantitatively by the terminal data \((y(T,\cdot), y_\Gamma(T,\cdot))\). More precisely, the main result of this paper is the following interpolation inequality, which is the key to proving the backward uniqueness and the desired stability estimate for the solutions of equation \eqref{eqq1.1}.
	\begin{thm}\label{thmm1.1}
		There exist large \(\beta_i > 0\) \((i=1,2)\) and \(M_0 > 0\) such that for any initial state \((y_0, y_{\Gamma,0}) \in L^2_{\mathcal{F}_0}(\Omega; \mathbb{L}^2)\) and \(t_0 \in (0,T]\), one can find constants \(\kappa \in (0,1)\) and \(C > 0\) depending on \(t_0\), \(T\), and a Nash equilibrium \((v^*_1, v^*_2) \in \mathcal{V}_1 \times \mathcal{V}_2\) such that the solution \((y, y_{\Gamma})\) of \eqref{eqq1.1} satisfies that
		\begin{align}\label{interine2.3}
			\begin{aligned}
				&\,\mathbb{E} \int_G y^2(t_0,x) \, dx + \mathbb{E} \int_\Gamma y_\Gamma^2(t_0,x) \,  d\sigma \\
				&\leq C \left( \mathbb{E} \int_Q y^2 \, dx\,dt + \mathbb{E} \int_\Sigma y_\Gamma^2 \,  d\sigma\,dt + M_0 \right)^{1-\kappa} \times\\
				&\hspace{.45cm}\left( \mathbb{E} \int_G y^2(T,x) \, dx + \mathbb{E} \int_\Gamma y_\Gamma^2(T,x) \,  d\sigma + M_1 \right)^{\kappa},
			\end{aligned}
		\end{align}
		where \(M_1 = \mathbb{E}\displaystyle\int_Q \left(|y_{1,d}|^2 \chi_{G_{1,d}} + |y_{2,d}|^2 \chi_{G_{2,d}}\right) \, dx \, dt\).
	\end{thm}
	From the theorem above, we can readily derive two important results. First, we deduce the following backward uniqueness result for the solutions of the equation \eqref{eqq1.1}.
	
	\begin{cor}\label{corol1}
		Under the assumptions of Theorem \ref{thmm1.1}, if 
		\begin{align*}
			y(T,\cdot) = 0 \quad \textnormal{in} \;\; G, \quad \mathbb{P}\textnormal{-a.s.} \quad\textnormal{and} \quad y_\Gamma(T,\cdot) = 0 \quad \textnormal{on} \;\; \Gamma, \quad \mathbb{P}\textnormal{-a.s.},
		\end{align*}
		and if the desired targets \(y_{i,d} = 0\) \((i=1,2)\) in \(Q\), $\mathbb{P}\textnormal{-a.s.}$, then there exists a Nash equilibrium \((v^*_1, v^*_2) \in \mathcal{V}_1 \times \mathcal{V}_2\) such that the corresponding solution \((y, y_\Gamma)\) of \eqref{eqq1.1} satisfies for all \(t \in (0,T]\)
		\begin{align*}
			y(t,\cdot) = 0 \quad \textnormal{in} \;\; G, \quad \mathbb{P}\textnormal{-a.s.} \quad\textnormal{and} \quad y_\Gamma(t,\cdot) = 0 \quad \textnormal{on} \;\; \Gamma, \quad \mathbb{P}\textnormal{-a.s.}
		\end{align*}
	\end{cor}
	
	Secondly, we obtain the following conditional stability estimate for solutions of \eqref{eqq1.1}. Let us first define the following set: For any \(M > 0\)
	\[
	\mathcal{U}_M = \left\{ (f,f_\Gamma) \in L^2_{\mathcal{F}_0}(\Omega; \mathbb{L}^2) \;\big|\; |(f,f_\Gamma)|_{L^2_{\mathcal{F}_0}(\Omega; \mathbb{L}^2)} \leq M \right\}.
	\]
	
	\begin{cor}\label{corol2}
		Under the assumptions of Theorem \ref{thmm1.1}, for any \((y_0, y_{\Gamma,0}) \in \mathcal{U}_M\) and \(t_0 \in (0,T]\), there exists a Nash equilibrium \((v^*_1, v^*_2) \in \mathcal{V}_1 \times \mathcal{V}_2\),  \(\kappa \in (0,1)\) and \(C > 0\) depending on \(t_0\), \(T\) such that the associated solution \((y, y_\Gamma)\) of equation \eqref{eqq1.1} satisfies 
		\begin{align*}
			\mathbb{E} \int_G y^2(t_0, x) \, dx + \mathbb{E} \int_\Gamma y_\Gamma^2(t_0, x) \,  d\sigma &\leq C(M, M_0) \left( \mathbb{E} \int_G y^2(T, x) \, dx + \mathbb{E} \int_\Gamma y_\Gamma^2(T, x) \,  d\sigma + M_1 \right)^{\kappa},
		\end{align*}
		where \(M_1 = \mathbb{E} \displaystyle\int_Q \left(|y_{1,d}|^2 \chi_{G_{1,d}} + |y_{2,d}|^2 \chi_{G_{2,d}}\right) \, dx \, dt\).
	\end{cor}
	
	\begin{rmk}
		\begin{enumerate}[1.]
			\item For simplicity, we use only two controls, \(v_1\) and \(v_2\), in this work. However, similar computations remain valid for any number of controls. Furthermore, in the absence of controls, our multi-objective inverse problem naturally reduces to the classical inverse problem. Thus, one can also obtain backward uniqueness and a conditional stability estimate for the initial conditions of equation \eqref{eqq1.1} (with \(v_1 \equiv v_2 \equiv 0\)).
			\item If we take \(a_2 \equiv b_2 \equiv 0\) in equation \eqref{eqq1.1} and assume all the coefficients \(a_1, b_1, v_1, v_2\), and \((y_0, y_{\Gamma,0})\) are deterministic functions, i.e., \(a_1 \in L^\infty(Q)\), \(b_1 \in L^\infty(\Sigma)\), \(v_1 \in L^2(0,T; L^2(G_1))\), \(v_2 \in L^2(0,T; L^2(G_2))\), and \((y_0, y_{\Gamma,0}) \in \mathbb{L}^2\), then we can deduce results such as backward uniqueness and a conditional stability estimate for the initial conditions under the Nash strategy (similar to those provided in Corollaries \ref{corol1} and \ref{corol2}) for the following deterministic heat equation with dynamic boundary conditions
			\begin{equation*}
				\begin{cases}
					\begin{array}{ll}
						\partial_t y - \Delta y  = a_1 y + \chi_{G_1}(x) v_1 + \chi_{G_2}(x) v_2 & \textnormal{in } Q,\\
						\partial_t y_\Gamma - \Delta_\Gamma y_\Gamma  + \partial_\nu y  = b_1 y_\Gamma  & \textnormal{on } \Sigma,\\
						y_\Gamma = y \vert_\Gamma & \textnormal{on } \Sigma,\\
						(y, y_\Gamma) \vert_{t=0} = (y_0, y_{\Gamma,0}) & \textnormal{in } G \times \Gamma.
					\end{array}
				\end{cases}
			\end{equation*}
		\end{enumerate}
	\end{rmk}
	\section{Nash Rquilibrium}\label{secc3}
	Since the functionals \(J_1\) and \(J_2\) are differentiable and convex, the pair \((v^{\star}_1, v^{\star}_2) \in \mathcal{V}_1 \times \mathcal{V}_2\) is a Nash equilibrium for \(J_1\) and \(J_2\), respectively, if and only if
	\begin{equation}\label{NE7}
		\left\{
		\begin{array}{ll}
			J_1'(v^{\star}_1, v^{\star}_2)(v,0) = 0, & \forall \, v \in \mathcal{V}_1, \\\\
			J_2'(v^{\star}_1, v^{\star}_2)(0,v) = 0, & \forall \, v \in \mathcal{V}_2,
		\end{array}
		\right.
	\end{equation}
	where \(J_1'\) and \(J_2'\) denote the Fréchet derivatives of \(J_1\) and \(J_2\), respectively. Following the same ideas as in \cite[Section 4]{BEMOstoch}, the results regarding the existence, uniqueness, and characterization of the Nash equilibrium are as follows.
	\begin{prop}\label{prop2.1}
		For sufficiently large \(\beta_i > 0\) \((i = 1, 2)\), there exists a unique Nash equilibrium \((v^{\star}_1, v^{\star}_2) \in \mathcal{V}_1 \times \mathcal{V}_2\) for \(J_1\) and \(J_2\), respectively. Moreover,
		\begin{equation*}
			v^{\star}_i = -\frac{1}{\beta_i} z^i \big|_{(0,T) \times G_i}, \qquad i = 1, 2,
		\end{equation*}
		where \(\left(z^i, z^i_\Gamma, Z^i, \widehat{Z}^i\right)\) is the solution of the following backward stochastic reaction-diffusion equation
		\begin{equation}\label{eqq4.701}
			\begin{cases}
				\begin{array}{ll}
					dz^i + \Delta z^i \, dt = \big[-a_1 z^i - a_2 Z^i - \alpha_i (y - y_{i,d}) \chi_{G_{i,d}}(x)\big] \, dt + Z^i \, dW(t) & \textnormal{in } Q, \\ 
					dz^i_\Gamma + \Delta_\Gamma z^i_\Gamma \, dt - \partial_\nu z^i \, dt = \big[-b_1 z^i_\Gamma - b_2 \widehat{Z}^i\big] \, dt + \widehat{Z}^i \, dW(t) & \textnormal{on } \Sigma, \\ 
					z^i_\Gamma = z^i|_\Gamma & \textnormal{on } \Sigma, \\ 
					(z^i, z^i_\Gamma)|_{t=T} = (0, 0), \qquad i = 1, 2, & \textnormal{in } G \times \Gamma.
				\end{array}
			\end{cases}
		\end{equation}
	\end{prop}
	
	According to Proposition \ref{prop2.1}, we now conclude by proving the interpolation inequality \eqref{interine2.3} for the solutions of the following coupled forward-backward stochastic system, derived from equations \eqref{eqq1.1} and \eqref{eqq4.701}
	
	\begin{equation}\label{eqq4.7}
		\begin{cases}
			\begin{array}{ll}
				dy - \Delta y \,dt = \left[a_1 y - \frac{1}{\beta_1} \chi_{G_1}(x) z^1 - \frac{1}{\beta_2} \chi_{G_2}(x) z^2\right] \,dt + a_2 y \,dW(t) & \textnormal{in } Q, \\
				dy_\Gamma - \Delta_\Gamma y_\Gamma \,dt + \partial_\nu y \,dt = b_1 y_\Gamma \,dt + b_2 y_\Gamma \,dW(t) & \textnormal{on } \Sigma, \\
				dz^{i} + \Delta z^{i} \,dt = \left[-a_1 z^{i} - a_2 Z^{i} - \alpha_i (y - y_{i,d}) \chi_{G_{i,d}}(x)\right] \,dt + Z^{i} \,dW(t) & \textnormal{in } Q, \\
				dz^{i}_{\Gamma} + \Delta_\Gamma z^{i}_\Gamma \,dt - \partial_\nu z^{i} \,dt = \left[-b_1 z^{i}_{\Gamma} - b_2 \widehat{Z}^{i}\right] \,dt + \widehat{Z}^{i} \,dW(t) & \textnormal{on } \Sigma, \\
				y_\Gamma = y|_\Gamma,\quad z^i_{\Gamma} = z^{i}|_\Gamma & \textnormal{on } \Sigma, \\
				(y, y_\Gamma)|_{t=0} = (y_0, y_{\Gamma, 0}) & \textnormal{in } G \times \Gamma, \\
				(z^i, z^{i}_\Gamma)|_{t=T} = (0, 0), \quad i=1,2 & \textnormal{in } G \times \Gamma.
			\end{array}
		\end{cases}
	\end{equation}
	
	\section{Carleman Estimates for Stochastic Parabolic Equations}\label{sec3}
	This section is devoted to proving new Carleman estimates for forward and backward stochastic parabolic equations with dynamic boundary conditions. We will employ a weighted identity method for this purpose. To begin, we introduce the following weight functions:
	$$
	\varphi(t) = e^{\lambda t}, \qquad \theta = e^{s \varphi},
	$$
	where \(t \in (0, +\infty)\) and \(\lambda, s > 0\) are sufficiently large parameters.
	
	Let us first prove two weighted identities for the stochastic parabolic operator \( ``du - b \Delta u \, dt "\) and the boundary stochastic parabolic operator $``du_\Gamma - b \Delta_\Gamma u_\Gamma \, dt + b \partial_\nu u \, dt"$, where \( u_\Gamma \) is the trace of \( u \) and $b\in\mathbb{R}$. These weighted identities will play a crucial role in the proof of our Carleman estimates.
	\begin{prop}\label{prop3.1}
		Let \(b \in \mathbb{R}\). Assume that \((u,u_\Gamma)\) is an \(\mathbb{H}^2\)-valued continuous semi-martingale, and define \((h,h_\Gamma) = \theta (u,u_\Gamma)\).
		\begin{enumerate}[1.]
			\item For any \(t \in [0,T]\) and a.e. \((x,\omega) \in G \times \Omega\), the following weighted identity holds
			\begin{align}\label{iden1}
				\begin{aligned}
					& -\theta \left(b \Delta h + s \lambda \varphi h\right) \left( du - b \Delta u \, dt \right) + \frac{1}{4} b \lambda \theta h \left( du - b \Delta u \, dt \right) \\
					& = -b \, \textnormal{div} \left[ \nabla h \, dh + \frac{1}{4} b \lambda h \nabla h \, dt \right] + \frac{1}{2} d \left(b |\nabla h|^2 - s \lambda \varphi h^2 + \frac{1}{4} b \lambda h^2 \right) \\
					& \quad - \frac{1}{2} b |d \nabla h|^2 + \frac{1}{4} b^2 \lambda |\nabla h|^2 \, dt + \frac{1}{2} s \lambda^2 \varphi h^2 \, dt - \frac{1}{4} b s \lambda^2 \varphi h^2 \, dt \\
					& \quad + \frac{1}{2} s \lambda \varphi (dh)^2 - \frac{1}{8} b \lambda (dh)^2 + \left(b \Delta h + s \lambda \varphi h\right)^2 \, dt.
				\end{aligned}
			\end{align}
			
			\item For any \(t \in [0,T]\) and a.e. \((x,\omega) \in \Gamma \times \Omega\), we have the following identity
			\begin{align}\label{iden2}
				\begin{aligned}
					& -\theta \left(b \Delta_\Gamma h_\Gamma + s \lambda \varphi h_\Gamma\right) \left( du_\Gamma - b \Delta_\Gamma u_\Gamma \, dt - b \partial_\nu u \, dt \right) \\
					& + \frac{1}{4} b \lambda \theta h_\Gamma \left( du_\Gamma - b \Delta_\Gamma u_\Gamma \, dt + b \partial_\nu u \, dt \right) \\
					& = -b \, \textnormal{div}_\Gamma \left[ \nabla_\Gamma h_\Gamma \, dh_\Gamma + \frac{1}{4} b \lambda h_\Gamma \nabla_\Gamma h_\Gamma \, dt \right] + \frac{1}{2} d \left(b |\nabla_\Gamma h_\Gamma|^2 - s \lambda \varphi h_\Gamma^2 + \frac{1}{4} b \lambda h_\Gamma^2 \right) \\
					& \quad - \frac{1}{2} b |d \nabla_\Gamma h_\Gamma|^2 + \frac{1}{4} b^2 \lambda |\nabla_\Gamma h_\Gamma|^2 \, dt + \frac{1}{2} s \lambda^2 \varphi h_\Gamma^2 \, dt - \frac{1}{4} b s \lambda^2 \varphi h_\Gamma^2 \, dt \\
					& \quad + \frac{1}{2} s \lambda \varphi (dh_\Gamma)^2 - \frac{1}{8} b \lambda (dh_\Gamma)^2 + \left(b \Delta_\Gamma h_\Gamma + s \lambda \varphi h_\Gamma\right)^2 \, dt \\
					& \quad + b \left(b \Delta_\Gamma h_\Gamma + s \lambda \varphi h_\Gamma\right) \partial_\nu h \, dt + \frac{1}{4} b^2 \lambda h_\Gamma \partial_\Gamma h \, dt.
				\end{aligned}
			\end{align}
		\end{enumerate}
	\end{prop}
	\begin{proof}
		We follow some ideas from the proof of \cite[Proposition 2.1]{StIP2}.
		
		\begin{enumerate}[1.]
			\item To compute the terms on the left-hand side of the identity \eqref{iden1}, we define
			\[
			I = -\theta \left(b \Delta h + s \lambda \varphi h\right) \left( du - b \Delta u \, dt \right),
			\qquad J = \frac{1}{4} b \lambda \theta h \left( du - b \Delta u \, dt \right).
			\]
			Using Itô's formula, we have 
			\begin{align*}
				I &= -\theta \left(b \Delta h + s \lambda \varphi h\right) \left( du - b \Delta u \, dt \right) \\
				&= -\left(b \Delta h + s \lambda \varphi h\right) \left(dh - s \lambda \varphi h \, dt - b \Delta h \, dt\right) \\
				&= -b \Delta h \, dh - s \lambda \varphi h \, dh + \left(b \Delta h + s \lambda \varphi h\right)^2 \, dt \\
				&= -b \, \textnormal{div} (\nabla h \, dh) + \frac{1}{2} b \, d |\nabla h|^2 - \frac{1}{2} b |d \nabla h|^2 + \frac{1}{2} s \lambda^2 \varphi h^2 \, dt \\
				&\quad - \frac{1}{2} d (s \lambda \varphi h^2) + \frac{1}{2} s \lambda \varphi (dh)^2 + \left(b \Delta h + s \lambda \varphi h\right)^2 \, dt.
			\end{align*}
			Similarly to $I$, we get
			\begin{align*}
				J &= \frac{1}{4} b \lambda \theta h \left( du - b \Delta u \, dt \right) \\
				&= \frac{1}{4} b \lambda h \left(dh - s \lambda \varphi h \, dt - b \Delta h \, dt\right) \\
				&= \frac{1}{4} b \lambda h \, dh - \frac{1}{4} b \lambda h \, (s \lambda \varphi h \, dt) - \frac{1}{4} b \lambda h \, (b \Delta h \, dt) \\
				&= \frac{1}{4} b \lambda h \, dh - \frac{1}{4} b s \lambda^2 \varphi h^2 \, dt - \frac{1}{4} b^2 \lambda \, \textnormal{div} (h \nabla h) \, dt + \frac{1}{4} b^2 \lambda |\nabla h|^2 \, dt.
			\end{align*}
			Combining \(I\) and \(J\), we obtain the identity \eqref{iden1}.
			\item For the boundary case, notice that
			\begin{align*}
				& -\theta \left(b \Delta_\Gamma h_\Gamma + s \lambda \varphi h_\Gamma\right) \left( du_\Gamma - b \Delta_\Gamma u_\Gamma \, dt - b \partial_\nu u \, dt \right) \\
				&\quad + \frac{1}{4} b \lambda \theta h_\Gamma \left( du_\Gamma - b \Delta_\Gamma u_\Gamma \, dt + b \partial_\nu u \, dt \right) \\
				&= I_\Gamma + J_\Gamma + b \theta \left(b \Delta_\Gamma h_\Gamma + s \lambda \varphi h_\Gamma\right) \partial_\nu u \, dt + \frac{1}{4} b^2 \lambda \theta h_\Gamma \partial_\Gamma u \, dt,
			\end{align*}
			where
			\[
			I_\Gamma = -\theta \left(b \Delta_\Gamma h_\Gamma + s \lambda \varphi h_\Gamma\right) \left( du_\Gamma - b \Delta_\Gamma u_\Gamma \, dt \right),\qquad
			J_\Gamma = \frac{1}{4} b \lambda \theta h_\Gamma \left( du_\Gamma - b \Delta_\Gamma u_\Gamma \, dt \right).
			\]
			Since the function \(\theta\) is $x$-independent, similar computations as for \eqref{iden1} apply to \(I_\Gamma\) and \(J_\Gamma\). Thus, we derive the identity \eqref{iden2}.
		\end{enumerate}
	\end{proof}
	\subsection{Carleman Estimate for Forward Stochastic Parabolic Equations}
	
	Let us consider the following general forward stochastic parabolic equation
	\begin{equation}\label{eqqgfr}
		\begin{cases}
			\begin{array}{ll}
				dz - \Delta z \, dt = f_1 \, dt + g_1 \, dW(t) & \textnormal{in } Q, \\
				dz_\Gamma - \Delta_\Gamma z_\Gamma \, dt + \partial_\nu z \, dt = f_2 \, dt + g_2 \, dW(t) & \textnormal{on } \Sigma, \\
				z_\Gamma = z|_\Gamma & \textnormal{on } \Sigma, \\
				(z, z_\Gamma)|_{t=0} = (z_0, z_{\Gamma,0}) & \textnormal{in } G \times \Gamma,
			\end{array}
		\end{cases}
	\end{equation}
	where \((z_0, z_{\Gamma,0}) \in L^2_{\mathcal{F}_0}(\Omega; \mathbb{L}^2)\) is the initial state, \(f_1 \in L^2_\mathcal{F}(0,T; L^2(G))\), \(g_1 \in L^2_\mathcal{F}(0,T; H^1(G))\), \(f_2 \in L^2_\mathcal{F}(0,T; L^2(\Gamma))\), and \(g_2 \in L^2_\mathcal{F}(0,T; H^1(\Gamma))\). We have the following Carleman estimate.
	\begin{thm}\label{thmmS33.1}
		There exists a large constant \( s_1 >0 \) such that for any \((z_0, z_{\Gamma,0}) \in L^2_{\mathcal{F}_0}(\Omega; \mathbb{L}^2)\), \( f_1 \in L^2_\mathcal{F}(0,T; L^2(G)) \), \( g_1 \in L^2_\mathcal{F}(0,T; H^1(G)) \), \( f_2 \in L^2_\mathcal{F}(0,T; L^2(\Gamma)) \), and \( g_2 \in L^2_\mathcal{F}(0,T; H^1(\Gamma)) \), the associated solution \((z, z_\Gamma)\) of \eqref{eqqgfr} satisfies that
		\begin{align}\label{Carlem1}
			\begin{aligned}
				&\, s \lambda^2 \mathbb{E} \int_Q \theta^2 \varphi z^2 \, dx \, dt + \lambda \mathbb{E} \int_Q \theta^2 |\nabla z|^2 \, dx \, dt + s \lambda \mathbb{E} \int_Q \theta^2 \varphi g_1^2 \, dx \, dt \\
				& + s \lambda^2 \mathbb{E} \int_\Sigma \theta^2 \varphi z_\Gamma^2 \,  d\sigma \, dt + \lambda \mathbb{E} \int_\Sigma \theta^2 |\nabla_\Gamma z_\Gamma|^2 \,  d\sigma \, dt + s \lambda \mathbb{E} \int_\Sigma \theta^2 \varphi g_2^2 \,  d\sigma \, dt \\
				& \leq C \, \mathbb{E} \bigg[ \theta^2(0) |\nabla z(0)|_{L^2(G)}^2 + s \lambda \varphi(T) \theta^2(T) |z(T)|_{L^2(G)}^2 + \theta^2(0) |\nabla_\Gamma z_\Gamma(0)|_{L^2(\Gamma)}^2  \\
				& \hspace{1.2cm}  + s \lambda \varphi(T) \theta^2(T) |z_\Gamma(T)|_{L^2(\Gamma)}^2 + \int_Q \theta^2 (f_1^2 + |\nabla g_1|^2) \, dx \, dt  \\
				& \hspace{1.2cm}  + \int_\Sigma \theta^2 (f_2^2 + |\nabla_\Gamma g_2|^2) \,  d\sigma \, dt \bigg],
			\end{aligned}
		\end{align}
		for all \( s \geq s_1 \) and all \( \lambda>0 \).
	\end{thm}
	\begin{proof}
		Using the identity \eqref{iden1} with \( b = 1 \), \( u = z \), and then integrating over \( Q \) and taking the expectation on both sides, we obtain that
		\begin{align}\label{a3.4}
			\begin{aligned}
				&-\mathbb{E}\int_Q \theta \left(\Delta h + s \lambda \varphi h\right) \left(dz - \Delta z \, dt\right) \, dx + \frac{1}{4} \lambda \mathbb{E} \int_Q \theta h \left(dz - \Delta z \, dt\right) \, dx \\
				&= -\mathbb{E} \int_Q \textnormal{div} \left[\nabla h \, dh + \frac{1}{4} \lambda h \nabla h \, dt \right] \, dx + \frac{1}{2} \mathbb{E} \int_Q d \left(|\nabla h|^2 - s \lambda \varphi h^2 + \frac{1}{4} \lambda h^2 \right) \, dx \\
				&\quad - \frac{1}{2} \mathbb{E} \int_Q |d \nabla h|^2 \, dx + \frac{1}{4} \lambda \mathbb{E} \int_Q |\nabla h|^2 \, dx \, dt + \frac{1}{4} s \lambda^2 \mathbb{E} \int_Q \varphi h^2 \, dx \, dt \\
				&\quad + \frac{1}{2} s \lambda \mathbb{E} \int_Q \varphi (dh)^2 \, dx - \frac{1}{8} \lambda \mathbb{E} \int_Q (dh)^2 \, dx + \mathbb{E} \int_Q (\Delta h + s \lambda \varphi h)^2 \, dx \, dt.
			\end{aligned}
		\end{align}
		Let us now estimate all the terms on the right-hand side of \eqref{a3.4}. We first have
		\begin{align}\label{eqqE1}
			\begin{aligned}
				- \mathbb{E} \int_Q \textnormal{div} \left[\nabla h \, dh + \frac{1}{4} \lambda h \nabla h \, dt \right] \, dx &= -\mathbb{E} \int_\Sigma \left(\partial_\nu h \, dh + \frac{1}{4} \lambda h \partial_\nu h \, dt \right) \,  d\sigma \\
				&= -s \lambda \mathbb{E} \int_\Sigma \theta^2 \varphi z_\Gamma \partial_\nu z \,  d\sigma \, dt - \mathbb{E} \int_\Sigma \theta^2 \Delta_\Gamma z_\Gamma \partial_\nu z \,  d\sigma \, dt \\
				&\quad + \mathbb{E} \int_\Sigma \theta^2 |\partial_\nu z|^2 \,  d\sigma \, dt - \mathbb{E} \int_\Sigma \theta^2 \partial_\nu z \, f_2 \,  d\sigma \, dt \\
				&\quad - \frac{1}{4} \lambda \mathbb{E} \int_\Sigma \theta^2 z_\Gamma \partial_\nu z \,  d\sigma \, dt.
			\end{aligned}
		\end{align}
		It is not difficult to see that there exists a large \( s_1>0 \) such that for all $s\geq s_1$
		\begin{align}\label{eqqE2}
			\begin{aligned}
				\frac{1}{2} \mathbb{E} \int_Q d \left(|\nabla h|^2 - s \lambda \varphi h^2 + \frac{1}{4} \lambda h^2 \right) \, dx &\geq -C \mathbb{E} \int_G \left[ |\nabla h(0)|^2 + s \lambda \varphi(T) h^2(T) \right] \, dx \\
				&\geq -C \mathbb{E} \int_G \left[ \theta^2(0) |\nabla z(0)|^2 + s \lambda \varphi(T) \theta^2(T) z^2(T) \right] \, dx.
			\end{aligned}
		\end{align}
		We also have 
		\begin{align}\label{eqqE3}
			-\frac{1}{2} \mathbb{E} \int_Q |d \nabla h|^2 \, dx + \frac{1}{4} \lambda \mathbb{E} \int_Q |\nabla h|^2 \, dx \, dt &= -\frac{1}{2} \mathbb{E} \int_Q \theta^2 |\nabla g_1|^2 \, dx \, dt + \frac{1}{4} \lambda \mathbb{E} \int_Q \theta^2 |\nabla z|^2 \, dx \, dt.
		\end{align}
		Taking a large \( s_1 \), we get
		\begin{align}\label{eqqE4}
			\begin{aligned}
				\frac{1}{2} s \lambda \mathbb{E} \int_Q \varphi (dh)^2 \, dx - \frac{1}{8} \lambda \mathbb{E} \int_Q (dh)^2 \, dx &= \frac{1}{2} s \lambda \mathbb{E} \int_Q \theta^2 \varphi g_1^2 \, dx \, dt - \frac{1}{8} \lambda \mathbb{E} \int_Q \theta^2 g_1^2 \, dx \, dt \\
				&\geq C s \lambda \mathbb{E} \int_Q \theta^2 \varphi g_1^2 \, dx \, dt.
			\end{aligned}
		\end{align}
		Let us now estimate the terms on the left-hand side of \eqref{a3.4}. We first note that
		\begin{align}\label{eqqE5}
			\begin{aligned}
				-\mathbb{E}\int_Q \theta \left(\Delta h + s \lambda \varphi h\right) \left(dz - \Delta z \, dt\right) \, dx &= -\mathbb{E}\int_Q \theta \left(\Delta h + s \lambda \varphi h\right) f_1 \, dx \, dt \\
				&\leq \mathbb{E}\int_Q \left(\Delta h + s \lambda \varphi h\right)^2 \, dx \, dt + \frac{1}{4} \mathbb{E}\int_Q \theta^2 f_1^2 \, dx \, dt.
			\end{aligned}
		\end{align}
		For a large \( s_1 \), we have that
		\begin{align}\label{eqqE6}
			\begin{aligned}
				\frac{1}{4} \lambda \mathbb{E}\int_Q \theta h \left(dz - \Delta z \, dt\right) \, dx &= \frac{1}{4} \lambda \mathbb{E}\int_Q \theta h f_1 \, dx \, dt \\
				&\leq \frac{1}{4} \lambda^2 \mathbb{E}\int_Q \theta^2 \varphi z^2 \, dx \, dt + \frac{1}{4} \mathbb{E}\int_Q \theta^2 f_1^2 \, dx \, dt.
			\end{aligned}
		\end{align}
		Combining \eqref{a3.4}, \eqref{eqqE1}, \eqref{eqqE2}, \eqref{eqqE3}, \eqref{eqqE4}, \eqref{eqqE5}, \eqref{eqqE6}, and taking a large enough \( s_1 \), we deduce 
		\begin{align}\label{estimm1}
			\begin{aligned}
				& C s \lambda^2 \mathbb{E}\int_Q \theta^2 \varphi z^2 \, dx \, dt + C \lambda \mathbb{E}\int_Q \theta^2 |\nabla z|^2 \, dx \, dt + C s \lambda \mathbb{E}\int_Q \theta^2 \varphi g_1^2 \, dx \, dt \\
				&\leq C \left[ \mathbb{E}\int_G \left[\theta^2(0) |\nabla z(0)|^2 + s \lambda \varphi(T) \theta^2(T) z^2(T)\right] \, dx \right] \\
				&\quad + \mathbb{E}\int_Q \theta^2 f_1^2 \, dx \, dt + \mathbb{E}\int_Q \theta^2 |\nabla g_1|^2 \, dx \, dt \\
				&\quad + s \lambda \mathbb{E}\int_\Sigma \theta^2 \varphi z_\Gamma \partial_\nu z \,  d\sigma \, dt + \mathbb{E}\int_\Sigma \theta^2 \Delta_\Gamma z_\Gamma \partial_\nu z \,  d\sigma \, dt\\
				&\quad  - \mathbb{E}\int_\Sigma \theta^2 |\partial_\nu z|^2 \,  d\sigma \, dt + \mathbb{E}\int_\Sigma \theta^2 \partial_\nu z f_2 \,  d\sigma \, dt + \frac{1}{4} \lambda \mathbb{E}\int_\Sigma \theta^2 z_\Gamma \partial_\nu z \,  d\sigma \, dt.
			\end{aligned}
		\end{align}
		On the other hand, applying the identity \eqref{iden2} with \( b=1 \), \( u_\Gamma=z_\Gamma \), and then integrating over \( \Sigma \) and taking the expectation on both sides, we get that
		\begin{align}\label{a.12}
			\begin{aligned}
				&-\mathbb{E}\int_\Sigma \theta \left(\Delta_\Gamma h_\Gamma + s \lambda \varphi h_\Gamma\right) \left(dz_\Gamma - \Delta_\Gamma z_\Gamma \, dt - \partial_\nu z \, dt\right) \,  d\sigma \\
				&+ \frac{1}{4} \lambda \mathbb{E}\int_\Sigma \theta h_\Gamma \left(dz_\Gamma - \Delta_\Gamma z_\Gamma \, dt + \partial_\nu z \, dt\right) \,  d\sigma \\
				&= -\mathbb{E}\int_\Sigma \textnormal{div}_\Gamma \left[\nabla_\Gamma h_\Gamma \, dh_\Gamma + \frac{1}{4} \lambda h_\Gamma \nabla_\Gamma h_\Gamma \, dt\right] \,  d\sigma \\
				&\quad + \frac{1}{2} \mathbb{E}\int_\Sigma d \left(|\nabla_\Gamma h_\Gamma|^2 - s \lambda \varphi h_\Gamma^2 + \frac{1}{4} \lambda h_\Gamma^2\right) \,  d\sigma - \frac{1}{2} \mathbb{E}\int_\Sigma |d \nabla_\Gamma h_\Gamma|^2 \,  d\sigma\\
				&\quad  + \frac{1}{4} \lambda \mathbb{E}\int_\Sigma |\nabla_\Gamma h_\Gamma|^2 \,  d\sigma \, dt + \frac{1}{4} s \lambda^2 \mathbb{E}\int_\Sigma \varphi h_\Gamma^2 \,  d\sigma \, dt\\
				&\quad  + \frac{1}{2} s \lambda \mathbb{E}\int_\Sigma \varphi (dh_\Gamma)^2 \,  d\sigma - \frac{1}{8} \lambda \mathbb{E}\int_\Sigma (dh_\Gamma)^2 \,  d\sigma\\
				&\quad  + \mathbb{E}\int_\Sigma (\Delta_\Gamma h_\Gamma + s \lambda \varphi h_\Gamma)^2 \,  d\sigma \, dt + \mathbb{E}\int_\Sigma \theta^2 \Delta_\Gamma z_\Gamma \, \partial_\nu z \,  d\sigma \, dt\\
				&\quad  + s \lambda \mathbb{E}\int_\Sigma \theta^2 \varphi z_\Gamma \partial_\nu z \,  d\sigma \, dt + \frac{1}{4} \lambda \mathbb{E}\int_\Sigma \theta^2 z_\Gamma \partial_\nu z \,  d\sigma \, dt.
			\end{aligned}
		\end{align}
		Let us estimate the terms on the right-hand side of \eqref{a.12}. We first note that
		\begin{align}\label{eqqEE21}
			-\mathbb{E}\int_\Sigma \textnormal{div}_\Gamma \left[\nabla_\Gamma h_\Gamma \, dh_\Gamma + \frac{1}{4} \lambda h_\Gamma \nabla_\Gamma h_\Gamma \, dt\right] \,  d\sigma = 0.
		\end{align}
		Taking a large $s_1$, we have 
		\begin{align}\label{eqqEE22}
			\begin{aligned}
				&\,\frac{1}{2} \mathbb{E}\int_\Sigma d \left(|\nabla_\Gamma h_\Gamma|^2 - s \lambda \varphi h_\Gamma^2 + \frac{1}{4} \lambda h_\Gamma^2\right) \,  d\sigma\\ &\geq -C \mathbb{E}\int_\Gamma \left[|\nabla_\Gamma h_\Gamma(0)|^2 + s \lambda \varphi(T) h_\Gamma^2(T)\right] \,  d\sigma \\
				&\geq -C \mathbb{E}\int_\Gamma \left[\theta^2(0) |\nabla_\Gamma z_\Gamma(0)|^2 + s \lambda \varphi(T) \theta^2(T) z_\Gamma^2(T)\right] \,  d\sigma.
			\end{aligned}
		\end{align}
		See that
		\begin{align}\label{eqqEE23}
			\begin{aligned}
				&-\frac{1}{2} \mathbb{E}\int_\Sigma |d \nabla_\Gamma h_\Gamma|^2 \,  d\sigma + \frac{1}{4} \lambda \mathbb{E}\int_\Sigma |\nabla_\Gamma h_\Gamma|^2 \,  d\sigma \, dt \\&= -\frac{1}{2} \mathbb{E}\int_\Sigma \theta^2 |\nabla_\Gamma g_2|^2 \,  d\sigma \, dt + \frac{1}{4} \lambda \mathbb{E}\int_\Sigma \theta^2 |\nabla_\Gamma z_\Gamma|^2 \,  d\sigma \, dt.
			\end{aligned}
		\end{align}
		For a large \( s_1 \), it holds that
		\begin{align}\label{eqqEE24}
			\begin{aligned}
				\frac{1}{2} s \lambda \mathbb{E}\int_\Sigma \varphi (dh_\Gamma)^2 \,  d\sigma - \frac{1}{8} \lambda \mathbb{E}\int_\Sigma (dh_\Gamma)^2 \,  d\sigma &= \frac{1}{2} s \lambda \mathbb{E}\int_\Sigma \theta^2 \varphi g_2^2 \,  d\sigma \, dt - \frac{1}{8} \lambda \mathbb{E}\int_\Sigma \theta^2 g_2^2 \,  d\sigma \, dt \\
				&\geq C s \lambda \mathbb{E}\int_\Sigma \theta^2 \varphi g_2^2 \,  d\sigma \, dt.
			\end{aligned}
		\end{align}
		To estimate the terms on the left-hand side of \eqref{a.12}, we proceed as in \eqref{eqqE5} and \eqref{eqqE6}, we obtain
		\begin{align}\label{eqqEE25}
			\begin{aligned}
				&-\mathbb{E}\int_\Sigma \theta\left(\Delta_\Gamma h_\Gamma + s\lambda\varphi h_\Gamma\right)\left(dz_\Gamma -\Delta_\Gamma z_\Gamma \, dt - \partial_\nu z \, dt\right) \,  d\sigma \\
				&\leq \mathbb{E}\int_\Sigma \left(\Delta_\Gamma h_\Gamma + s\lambda\varphi h_\Gamma\right)^2 \,  d\sigma \, dt + \frac{1}{4}\mathbb{E}\int_\Sigma \theta^2 (-2\partial_\nu z + f_2)^2 \,  d\sigma \, dt \\
				&= \mathbb{E}\int_\Sigma \left(\Delta_\Gamma h_\Gamma + s\lambda\varphi h_\Gamma\right)^2 \,  d\sigma \, dt + \mathbb{E}\int_\Sigma \theta^2 |\partial_\nu z|^2 \,  d\sigma \, dt - \mathbb{E}\int_\Sigma \theta^2 \partial_\nu z \, f_2 \,  d\sigma \, dt \\
				&\quad + \frac{1}{4} \mathbb{E}\int_\Sigma \theta^2 f_2^2 \,  d\sigma \, dt.
			\end{aligned}
		\end{align}
		By Young's inequality, we have that
		\begin{align}\label{eqqEE26}
			\begin{aligned}
				\frac{1}{4}\lambda \mathbb{E}\int_\Sigma \theta h_\Gamma \left(dz_\Gamma - \Delta_\Gamma z_\Gamma \, dt + \partial_\nu z \, dt\right) \,  d\sigma &\leq \frac{1}{4}\lambda^2 \mathbb{E}\int_\Sigma \theta^2 \varphi z_\Gamma^2 \,  d\sigma \, dt + \frac{1}{4} \mathbb{E}\int_\Sigma \theta^2 f_2^2 \,  d\sigma \, dt.
			\end{aligned}
		\end{align}
		Combining \eqref{a.12}, \eqref{eqqEE21}, \eqref{eqqEE22}, \eqref{eqqEE23}, \eqref{eqqEE24}, \eqref{eqqEE25}, and \eqref{eqqEE26}, and choosing \( s_1 \) sufficiently large, we conclude that
		\begin{align}\label{estimm2}
			\begin{aligned}
				& Cs\lambda^2 \mathbb{E}\int_\Sigma \theta^2 \varphi z_\Gamma^2 \,  d\sigma \, dt + C\lambda \mathbb{E}\int_\Sigma \theta^2 |\nabla_\Gamma z_\Gamma|^2 \,  d\sigma \, dt + Cs\lambda \mathbb{E}\int_\Sigma \theta^2 \varphi g_2^2 \,  d\sigma \, dt \\
				&\leq C\left[\mathbb{E}\int_\Gamma \left[\theta^2(0) |\nabla_\Gamma z_\Gamma(0)|^2 + s\lambda \varphi(T) \theta^2(T) z_\Gamma^2(T)\right] \,  d\sigma \right] \\
				&\quad + \mathbb{E}\int_\Sigma \theta^2 f_2^2 \,  d\sigma \, dt + \mathbb{E}\int_\Sigma \theta^2 |\nabla_\Gamma g_2|^2 \,  d\sigma \, dt - \mathbb{E}\int_\Sigma \theta^2 \Delta_\Gamma z_\Gamma \, \partial_\nu z \,  d\sigma \, dt \\
				&\quad - s\lambda \mathbb{E}\int_\Sigma \theta^2 \varphi z_\Gamma \partial_\nu z \,  d\sigma \, dt - \frac{1}{4} \lambda \mathbb{E}\int_\Sigma \theta^2 z_\Gamma \partial_\nu z \,  d\sigma \, dt + \mathbb{E}\int_\Sigma \theta^2 |\partial_\nu z|^2 \,  d\sigma \, dt \\
				&\quad - \mathbb{E}\int_\Sigma \theta^2 \partial_\nu z \, f_2 \,  d\sigma \, dt.
			\end{aligned}
		\end{align}
		Finally, adding \eqref{estimm1} and \eqref{estimm2}, we end up with
		\begin{align*}
			\begin{aligned}
				&s\lambda^2 \mathbb{E}\int_Q \theta^2 \varphi z^2 \, dx \, dt + \lambda \mathbb{E}\int_Q \theta^2 |\nabla z|^2 \, dx \, dt + s\lambda \mathbb{E}\int_Q \theta^2 \varphi g_1^2 \, dx \, dt \\
				&+ s\lambda^2 \mathbb{E}\int_\Sigma \theta^2 \varphi z_\Gamma^2 \,  d\sigma \, dt + \lambda \mathbb{E}\int_\Sigma \theta^2 |\nabla_\Gamma z_\Gamma|^2 \,  d\sigma \, dt + s\lambda \mathbb{E}\int_\Sigma \theta^2 \varphi g_2^2 \,  d\sigma \, dt \\
				&\leq C \bigg[ \mathbb{E}\int_G \left[\theta^2(0) |\nabla z(0)|^2 + s\lambda \varphi(T) \theta^2(T) z^2(T)\right] \, dx \\
				&\hspace{0.8cm} + \mathbb{E}\int_Q \theta^2 f_1^2 \, dx \, dt + \mathbb{E}\int_Q \theta^2 |\nabla g_1|^2 \, dx \, dt \\
				&\hspace{0.8cm} + \mathbb{E}\int_\Gamma \left[\theta^2(0) |\nabla_\Gamma z_\Gamma(0)|^2 + s\lambda \varphi(T) \theta^2(T) z_\Gamma^2(T)\right] \,  d\sigma \\
				&\hspace{0.8cm} + \mathbb{E}\int_\Sigma \theta^2 f_2^2 \,  d\sigma \, dt + \mathbb{E}\int_\Sigma \theta^2 |\nabla_\Gamma g_2|^2 \,  d\sigma \, dt \bigg],
			\end{aligned}
		\end{align*}
		which implies the desired Carleman estimate \eqref{Carlem1}. This concludes the proof of Theorem \ref{thmmS33.1}.
	\end{proof}
	\subsection{Carleman Estimate for Backward Stochastic Parabolic Equations}
	
	Consider the following general backward stochastic parabolic equation
	\begin{equation}\label{eqqgbc}
		\begin{cases}
			dz + \Delta z \, dt = F_1 \, dt + Z \, dW(t) & \textnormal{in } Q, \\
			d z_{\Gamma} + \Delta_\Gamma z_\Gamma \, dt - \partial_\nu z \, dt = F_2 \, dt + \widehat{Z} \, dW(t) & \textnormal{on } \Sigma, \\
			z_{\Gamma} = z|_\Gamma & \textnormal{on } \Sigma, \\
			(z, z_\Gamma)|_{t=T} = (z_T, z_{\Gamma, T}) & \textnormal{in } G \times \Gamma,
		\end{cases}
	\end{equation}
	where \((z_T, z_{\Gamma, T}) \in L^2_{\mathcal{F}_T}(\Omega; \mathbb{L}^2)\) is the terminal state, \(F_1 \in L^2_\mathcal{F}(0, T; L^2(G))\), and \(F_2 \in L^2_\mathcal{F}(0, T; L^2(\Gamma))\). We have the following Carleman estimate.
	
	\begin{thm}\label{thmmS32.1}
		Assume that \((z_T, z_{\Gamma, T}) = (0, 0)\). Then, there exists a large \(s_2>0\) such that for any \(F_1 \in L^2_\mathcal{F}(0, T; L^2(G))\) and \(F_2 \in L^2_\mathcal{F}(0, T; L^2(\Gamma))\), the corresponding solution \((z, z_\Gamma, Z, \widehat{Z})\) of \eqref{eqqgbc} satisfies that
		\begin{align}\label{carl2}
			\begin{aligned}
				& s\lambda^2 \mathbb{E} \int_Q \theta^2 \varphi z^2 \, dx \, dt + \lambda \mathbb{E} \int_Q \theta^2 |\nabla z|^2 \, dx \, dt + s\lambda \mathbb{E} \int_Q \theta^2 \varphi Z^2 \, dx \, dt \\
				& + s\lambda^2 \mathbb{E} \int_\Sigma \theta^2 \varphi z_\Gamma^2 \,  d\sigma \, dt + \lambda \mathbb{E} \int_\Sigma \theta^2 |\nabla_\Gamma z_\Gamma|^2 \,  d\sigma \, dt + s\lambda \mathbb{E} \int_\Sigma \theta^2 \varphi \widehat{Z}^2 \,  d\sigma \, dt \\
				& \leq C \left[ \mathbb{E} \int_Q \theta^2 F_1^2 \, dx \, dt + \mathbb{E} \int_\Sigma \theta^2 F_2^2 \,  d\sigma \, dt \right],
			\end{aligned}
		\end{align}
		for all \(s \geq s_2\) and all \(\lambda>0\).
	\end{thm}
	\begin{proof}
		Using the identity \eqref{iden1} with \( b = -1 \) and \( u = z \), we obtain that
		\begin{align}\label{abac3.4}
			\begin{aligned}
				&-\mathbb{E}\int_Q \theta \left(-\Delta h + s\lambda \varphi h\right) \left(dz + \Delta z \, dt\right) \, dx - \frac{1}{4}\lambda \mathbb{E}\int_Q \theta h \left(dz + \Delta z \, dt\right) \, dx \\
				&= \mathbb{E}\int_Q \textnormal{div} \left[\nabla h \, dh - \frac{1}{4}\lambda h \nabla h \, dt \right] \, dx - \frac{1}{2} \mathbb{E}\int_Q d \left(|\nabla h|^2 + s\lambda \varphi h^2 + \frac{1}{4}\lambda h^2 \right) \, dx \\
				&\quad + \frac{1}{2} \mathbb{E}\int_Q |d\nabla h|^2 \, dx + \frac{1}{4}\lambda \mathbb{E}\int_Q |\nabla h|^2 \, dx \, dt + \frac{3}{4}s\lambda^2 \mathbb{E}\int_Q \varphi h^2 \, dx \, dt \\
				&\quad + \frac{1}{2}s\lambda \mathbb{E}\int_Q \varphi (dh)^2 \, dx + \frac{1}{8}\lambda \mathbb{E}\int_Q (dh)^2 \, dx + \mathbb{E}\int_Q \left(-\Delta h + s\lambda \varphi h\right)^2 \, dx \, dt.
			\end{aligned}
		\end{align}
		Let us first estimate the terms on the right-hand side of \eqref{abac3.4}. Notice that
		\begin{align}\label{abacc1}
			\begin{aligned}
				\mathbb{E}\int_Q \textnormal{div} \left[\nabla h \, dh - \frac{1}{4}\lambda h \nabla h \, dt \right] \, dx &= s\lambda \mathbb{E}\int_\Sigma \theta^2 \varphi z_\Gamma \partial_\nu z \,  d\sigma \, dt - \mathbb{E}\int_\Sigma \theta^2 \Delta_\Gamma z_\Gamma \partial_\nu z \,  d\sigma \, dt \\
				&\quad + \mathbb{E}\int_\Sigma \theta^2 |\partial_\nu z|^2 \,  d\sigma \, dt + \mathbb{E}\int_\Sigma \theta^2 \partial_\nu z \, F_2 \,  d\sigma \, dt \\
				&\quad - \frac{1}{4}\lambda \mathbb{E}\int_\Sigma \theta^2 z_\Gamma \partial_\nu z \,  d\sigma \, dt.
			\end{aligned}
		\end{align}
		It is not difficult to see that there exists a large $s_2>0$ such that for all \( s\geq s_2 \)
		\begin{align}\label{abacc2}
			-\frac{1}{2}\mathbb{E}\int_Q d \left(|\nabla h|^2 + s\lambda \varphi h^2 + \frac{1}{4}\lambda h^2 \right) \, dx \geq -C \mathbb{E}\int_G \theta^2(T) |\nabla z(T)|^2 \, dx.
		\end{align}
		We also note that
		\begin{align}\label{abacc3}
			\begin{aligned}
				\frac{1}{2}\mathbb{E}\int_Q |d\nabla h|^2 \, dx \, dt + \frac{1}{4}\lambda \mathbb{E}\int_Q |\nabla h|^2 \, dx \, dt &\geq \frac{1}{4}\lambda \mathbb{E}\int_Q \theta^2 |\nabla z|^2 \, dx \, dt,
			\end{aligned}
		\end{align}
		and
		\begin{align}\label{abacc4}
			\begin{aligned}
				\frac{1}{2}s\lambda \mathbb{E}\int_Q \varphi (dh)^2 \, dx + \frac{1}{8}\lambda \mathbb{E}\int_Q (dh)^2 \, dx &= \frac{1}{2}s\lambda \mathbb{E}\int_Q \theta^2 \varphi Z^2 \, dx \, dt + \frac{1}{8}\lambda \mathbb{E}\int_Q \theta^2 Z^2 \, dx \, dt \\
				&\geq \frac{1}{2}s\lambda \mathbb{E}\int_Q \theta^2 \varphi Z^2 \, dx \, dt.
			\end{aligned}
		\end{align}
		For the terms on the left-hand side of \eqref{abac3.4}, we have that
		\begin{align}\label{abacc5}
			\begin{aligned}
				-\mathbb{E}\int_Q \theta \left(-\Delta h + s\lambda \varphi h \right) \left(dz + \Delta z \, dt\right) \, dx &= -\mathbb{E}\int_Q \theta \left(-\Delta h + s\lambda \varphi h \right) F_1 \, dx \, dt \\
				&\leq \mathbb{E}\int_Q \left(-\Delta h + s\lambda \varphi h\right)^2 \, dx \, dt + \frac{1}{4}\mathbb{E}\int_Q \theta^2 F_1^2 \, dx \, dt.
			\end{aligned}
		\end{align}
		Similarly to \eqref{abacc5}, we get
		\begin{align}\label{abacc6}
			\begin{aligned}
				-\frac{1}{4}\lambda \mathbb{E}\int_Q \theta h \left(dz + \Delta z \, dt\right) \, dx &= -\frac{1}{4}\lambda \mathbb{E}\int_Q \theta h F_1 \, dx \, dt \\
				&\leq \frac{1}{4}\lambda^2 \mathbb{E}\int_Q \theta^2 \varphi z^2 \, dx \, dt + \frac{1}{4} \mathbb{E}\int_Q \theta^2 F_1^2 \, dx \, dt.
			\end{aligned}
		\end{align}
		Combining \eqref{abac3.4}, \eqref{abacc1}, \eqref{abacc2}, \eqref{abacc3}, \eqref{abacc4}, \eqref{abacc5} and  \eqref{abacc6}, and choosing a large enough $s_2$, we conclude that
		\begin{align}\label{ESTIMATEE1}
			\begin{aligned}
				&Cs\lambda^2\mathbb{E}\int_Q \theta^2\varphi z^2 \,\, dx\,dt + C\lambda\mathbb{E}\int_Q \theta^2|\nabla z|^2 \,\, dx\,dt+Cs\lambda\mathbb{E}\int_Q\theta^2\varphi Z^2\,\, dx\,dt\\
				&\leq C\mathbb{E}\int_G \theta^2(T)|\nabla z(T)|^2\, dx+\mathbb{E}\int_Q\theta^2 F_1^2 \,\, dx\,dt\\
				&\quad-s\lambda\mathbb{E}\int_\Sigma \theta^2\varphi z_\Gamma \partial_\nu z\,  d\sigma\,dt+\mathbb{E}\int_\Sigma \theta^2\Delta_\Gamma z_\Gamma \partial_\nu z\,  d\sigma\,dt-\mathbb{E}\int_\Sigma \theta^2|\partial_\nu z|^2 \,  d\sigma\,dt\\
				&\quad-\mathbb{E}\int_\Sigma \theta^2\partial_\nu z F_2 \,  d\sigma\,dt+\frac{1}{4}\lambda\mathbb{E}\int_\Sigma \theta^2 z_\Gamma \partial_\nu z \,  d\sigma\,dt.
			\end{aligned}
		\end{align}
		On the other hand, applying the identity \eqref{iden2} with $b=-1$ and $u_\Gamma=z_\Gamma$,  we get that
		\begin{align}\label{abacca12}
			\begin{aligned}
				&-\mathbb{E}\int_\Sigma \theta\left(-\Delta_\Gamma h_\Gamma + s\lambda\varphi h_\Gamma\right)\left(dz_\Gamma +\Delta_\Gamma z_\Gamma\, dt+\partial_\nu z\,dt\right)\,  d\sigma\\
				&- \frac{1}{4}\lambda \mathbb{E}\int_\Sigma \theta h_\Gamma\left(dz_\Gamma +\Delta_\Gamma z_\Gamma\, dt-\partial_\nu z\,dt\right)\,  d\sigma\\
				&=\mathbb{E}\int_\Sigma \textnormal{div}_\Gamma\left[\nabla_\Gamma h_\Gamma\,dh_\Gamma-\frac{1}{4}\lambda h_\Gamma\nabla_\Gamma h_\Gamma\,dt\right]\,  d\sigma\\
				&\quad\,-\frac{1}{2}\mathbb{E}\int_\Sigma d\left(|\nabla_\Gamma h_\Gamma|^2+s\lambda\varphi h_\Gamma^2+\frac{1}{4}\lambda h_\Gamma^2\right)\,  d\sigma+\frac{1}{2}\mathbb{E}\int_\Sigma |d\nabla_\Gamma h_\Gamma|^2\,  d\sigma
				\\
				&\quad\,+\frac{1}{4}\lambda\mathbb{E}\int_\Sigma |\nabla_\Gamma h_\Gamma|^2\, d\sigma\,dt+\frac{3}{4}s\lambda^2\mathbb{E}\int_\Sigma \varphi h_\Gamma^2 \, d\sigma\,dt+\frac{1}{2}s\lambda\mathbb{E}\int_\Sigma \varphi (dh_\Gamma)^2\,  d\sigma\\
				&\quad\,+\frac{1}{8}\lambda\mathbb{E}\int_\Sigma (dh_\Gamma)^2\,  d\sigma+\mathbb{E}\int_\Sigma (-\Delta_\Gamma h_\Gamma + s\lambda\varphi h_\Gamma)^2\,\,d\sigma\,dt\\
				&\quad\,+\mathbb{E}\int_\Sigma \theta^2 \Delta_\Gamma z_\Gamma \,\partial_\nu z  \,  d\sigma\,dt-s\lambda\mathbb{E}\int_\Sigma \theta^2 \varphi z_\Gamma \partial_\nu z \,  d\sigma\,dt+\frac{1}{4}\lambda\mathbb{E}\int_\Sigma \theta^2 z_\Gamma\partial_\nu z \,  d\sigma\,dt.
			\end{aligned}
		\end{align}
		For the right hand side of \eqref{abacca12}. We first have that
		\begin{align}\label{aineq1}
			\mathbb{E}\int_\Sigma \textnormal{div}_\Gamma\left[\nabla_\Gamma h_\Gamma\,dh_\Gamma-\frac{1}{4}\lambda h_\Gamma\nabla_\Gamma h_\Gamma\,dt\right]\,  d\sigma=0.
		\end{align}
		It is easy to see that for a large enough $s_2$
		\begin{align}\label{aineq2}
			-\frac{1}{2}\mathbb{E}\int_\Sigma d\left(|\nabla_\Gamma h_\Gamma|^2+s\lambda\varphi h_\Gamma^2+\frac{1}{4}\lambda h_\Gamma^2\right)\, dx \geq -C\mathbb{E}\int_\Gamma \theta^2(T)|\nabla_\Gamma z_\Gamma(T)|^2 \,  d\sigma.
		\end{align}
		We also have that
		\begin{align}\label{aineq3}
			\begin{aligned}
				\frac{1}{2}\mathbb{E}\int_\Sigma|d\nabla_\Gamma h_\Gamma|^2\, d\sigma\,dt+\frac{1}{4}\lambda\mathbb{E}\int_\Sigma|\nabla_\Gamma h_\Gamma|^2\, d\sigma\,dt\geq \frac{1}{4}\lambda\mathbb{E}\int_\Sigma\theta^2|\nabla_\Gamma z_\Gamma|^2\, d\sigma\,dt,
			\end{aligned}
		\end{align}
		and
		\begin{align}\label{aineq4}
			\begin{aligned}
				\frac{1}{2}s\lambda\mathbb{E}\int_\Sigma\varphi (dh_\Gamma)^2\,  d\sigma+\frac{1}{8}\lambda\mathbb{E}\int_\Sigma(dh_\Gamma)^2\,  d\sigma\geq \frac{1}{2}s\lambda\mathbb{E}\int_\Sigma\theta^2\varphi \widehat{Z}^2\, d\sigma\,dt.
			\end{aligned}
		\end{align}
		We now estimate the terms on the left hand side of \eqref{abacca12}. From equation \eqref{eqqgbc}, we have 
		\begin{align}\label{aineq5}
			\begin{aligned}
				&-\mathbb{E}\int_\Sigma \theta\left(-\Delta_\Gamma h_\Gamma + s\lambda\varphi h_\Gamma\right)\left(dz_\Gamma +\Delta_\Gamma z_\Gamma\, dt+\partial_\nu z\,dt\right)\, d\sigma \\
				&\leq \mathbb{E}\int_\Sigma \left(-\Delta_\Gamma h_\Gamma + s\lambda\varphi h_\Gamma\right)^2 \,  d\sigma\,dt+\frac{1}{4}\mathbb{E}\int_\Sigma \theta^2 (2\partial_\nu z+F_2)^2 \,  d\sigma\,dt\\
				&= \mathbb{E}\int_\Sigma \left(-\Delta_\Gamma h_\Gamma + s\lambda\varphi h_\Gamma\right)^2 \,  d\sigma\,dt+\mathbb{E}\int_\Sigma \theta^2|\partial_\nu z|^2 \,  d\sigma\,dt+\mathbb{E}\int_\Sigma \theta^2 \partial_\nu z F_2 \,  d\sigma\,dt\\
				&\quad\,+\frac{1}{4}\mathbb{E}\int_\Sigma \theta^2 F_2^2 \, d\sigma\,dt.
			\end{aligned}
		\end{align}
		Notice that
		\begin{align}\label{aineq6}
			\begin{aligned}
				-\frac{1}{4}\lambda\mathbb{E}\int_\Sigma\theta h_\Gamma\left(dz_\Gamma+\Delta_\Gamma z_\Gamma\, dt-\partial_\nu z\,dt\right)\,  d\sigma &\leq  \frac{1}{4}\lambda^2\mathbb{E}\int_\Sigma \theta^2\varphi z_\Gamma^2 \,  d\sigma\,dt+\frac{1}{4}\mathbb{E}\int_\Sigma\theta^2 F_2^2 \,  d\sigma\,dt.
			\end{aligned}
		\end{align}
		Combining \eqref{abacca12}, \eqref{eqqEE21}, \eqref{eqqEE22}, \eqref{eqqEE23}, \eqref{eqqEE24}, \eqref{eqqEE25}, \eqref{eqqEE26}, and taking a large $s_2$, we deduce that
		\begin{align}\label{ESTIMATEE2}
			\begin{aligned}
				&Cs\lambda^2\mathbb{E}\int_\Sigma \theta^2\varphi z_\Gamma^2 \, d\sigma\,dt + C\lambda\mathbb{E}\int_\Sigma \theta^2|\nabla_\Gamma z_\Gamma|^2 \, d\sigma\,dt+Cs\lambda\mathbb{E}\int_\Sigma\theta^2 \varphi 
				\widehat{Z}^2\, d\sigma\,dt\\
				&\leq C\mathbb{E}\int_\Gamma \theta^2(T)|\nabla_\Gamma z_\Gamma(T)|^2 \,  d\sigma+\mathbb{E}\int_\Sigma\theta^2 F_2^2 \, d\sigma\,dt\\
				&\quad+s\lambda\mathbb{E}\int_\Sigma \theta^2 \varphi z_\Gamma \partial_\nu z \,  d\sigma\,dt-\mathbb{E}\int_\Sigma \theta^2 \Delta_\Gamma z_\Gamma \,\partial_\nu z  \,  d\sigma\,dt\\
				&\quad-\frac{1}{4}\lambda\mathbb{E}\int_\Sigma \theta^2 z_\Gamma\partial_\nu z \,  d\sigma\,dt+\mathbb{E}\int_\Sigma \theta^2|\partial_\nu z|^2 \,  d\sigma\,dt+\mathbb{E}\int_\Sigma \theta^2 \partial_\nu z F_2 \,  d\sigma\,dt.
			\end{aligned}
		\end{align}
		From \eqref{ESTIMATEE1} and \eqref{ESTIMATEE2}, we obtain
		\begin{align}\label{estmm340}
			\begin{aligned}
				&\,s\lambda^2\mathbb{E}\int_Q \theta^2\varphi z^2 \,\, dx\,dt + \lambda\mathbb{E}\int_Q \theta^2|\nabla z|^2 \,\, dx\,dt+s\lambda\mathbb{E}\int_Q\theta^2\varphi Z^2\,\, dx\,dt\\
				&+s\lambda^2\mathbb{E}\int_\Sigma \theta^2\varphi z_\Gamma^2 \, d\sigma\,dt + \lambda\mathbb{E}\int_\Sigma \theta^2|\nabla_\Gamma z_\Gamma|^2 \, d\sigma\,dt+s\lambda\mathbb{E}\int_\Sigma\theta^2 \varphi 
				\widehat{Z}^2\, d\sigma\,dt\\
				&\leq C\Bigg[\mathbb{E}\int_G \theta^2(T)|\nabla z(T)|^2 \, dx+\mathbb{E}\int_\Gamma \theta^2(T)|\nabla_\Gamma z_\Gamma(T)|^2 \,  d\sigma\\
				&\qquad\;+\mathbb{E}\int_Q\theta^2 F_1^2 \,\, dx\,dt+\mathbb{E}\int_\Sigma\theta^2 F_2^2 \, d\sigma\,dt\Bigg].
			\end{aligned}
		\end{align}
		From the well-posedness of \eqref{eqqgbc} with \((z_T, z_{\Gamma,T}) = (0,0)\), see for instance \cite{vanNervenWeis12}, we have 
		\begin{align*}
			|(z, z_\Gamma)|_{L^2_\mathcal{F}(\Omega; C([0,T]; \mathbb{H}^1))} \leq C \left( |F_1|_{L^2_\mathcal{F}(0,T; L^2(G))} + |F_2|_{L^2_\mathcal{F}(0,T; L^2(\Gamma))} \right).
		\end{align*}
		Then, it follows that
		\begin{align}\label{3.41estim}
			\begin{aligned}
				&\mathbb{E}\int_G \theta^2(T) |\nabla z(T)|^2 \, dx + \mathbb{E}\int_\Gamma \theta^2(T) |\nabla_\Gamma z_\Gamma(T)|^2 \,  d\sigma \\
				&\leq C \Bigg[ \mathbb{E}\int_Q \theta^2 F_1^2 \, dx \, dt + \mathbb{E}\int_\Sigma \theta^2 F_2^2 \,  d\sigma \, dt \Bigg].
			\end{aligned}
		\end{align}
		Combining \eqref{estmm340} and \eqref{3.41estim}, it is straightforward to deduce the desired Carleman estimate \eqref{carl2}. This completes the proof of Theorem \ref{thmmS32.1}.
	\end{proof}
	\section{Proof of the main results}\label{sec4}  
	In this section, we present the proof of the main result, i.e., Theorem \ref{thmm1.1}. We achieve this by combining our Carleman estimates \eqref{Carlem1} and \eqref{carl2} for the coupled forward-backward stochastic system \eqref{eqq4.7}. The proofs of Corollaries \ref{corol1} and \ref{corol2} are easy consequences of Theorem \ref{thmm1.1}.
	\begin{proof}[Proof of Theorem \ref{thmm1.1}]
		Let us choose a cut-off function \(\eta \in C^\infty([0,T]; [0,1])\) such that \(\eta \equiv 0\) in \([0, t_2]\) and \(\eta \equiv 1\) in \([t_1, T]\), where \(0 < t_2 < t_1 < t_0 < T\). Define 
		\[(\xi, \xi_\Gamma) = \eta(y, y_\Gamma),\quad\textnormal{and}\quad \left(p^i, p^i_\Gamma, P^i, \widehat{P}^i\right) = \eta\left(z^i, z^i_\Gamma, Z^i, \widehat{Z}^i\right),\;i=1,2,\]
		where \(\left((y, y_\Gamma), z^i, z^i_\Gamma, Z^i, \widehat{Z}^i\right)\) is the solution of \eqref{eqq4.7}. Then \(\left((\xi, \xi_\Gamma), p^i, p^i_\Gamma, P^i, \widehat{P}^i\right)\) is the solution of the following forward-backward stochastic system
		\begin{equation}\label{forbacksyst}
			\begin{cases}
				\begin{array}{ll}
					d\xi - \Delta \xi \,dt = \left[a_1\xi-\frac{1}{\beta_1}\chi_{G_1}(x) p^1-\frac{1}{\beta_2}\chi_{G_2}(x) p^2+\eta_t y\right] \,dt + a_2\xi\,dW(t)&\textnormal{in}\,\,Q,\\
					d\xi_\Gamma-\Delta_\Gamma \xi_\Gamma \,dt+\partial_\nu \xi \,dt = [b_1\xi_\Gamma+\eta_t y_\Gamma] \,dt+b_2\xi_\Gamma \,dW(t)&\textnormal{on}\,\,\Sigma,\\
					dp^{i}+\Delta p^{i} \,dt=\big[-a_1 p^{i}-a_2 P^{i}-\alpha_i\eta(y-y_{i,d})\chi_{G_{i,d}}(x)+\eta_t z^i\big] \,dt+P^{i} \,dW(t)&\textnormal{in}\,\,Q, \\ 
					dp^{i}_{\Gamma}+\Delta_\Gamma p^{i}_\Gamma \,dt-\partial_\nu p^i \,dt=\big[-b_1 p^{i}_{\Gamma}-b_2 \widehat{P}^{i}+\eta_t z^i_\Gamma\big] \,dt+\widehat{P}^{i} \,dW(t)&\textnormal{on}\,\,\Sigma, \\ 
					\xi_\Gamma=\xi\vert_\Gamma,\quad p^i_{\Gamma}=p^{i}|_\Gamma &\textnormal{on}\,\,\Sigma,\\
					(\xi, \xi_\Gamma)|_{t=0}=(0,0) &\textnormal{in}\,\, G\times \Gamma, \\
					(p^i, p^{i}_\Gamma)|_{t=T}=(0, 0),\qquad i=1,2,&\textnormal{in}\,\,G \times \Gamma.
				\end{array}
			\end{cases}
		\end{equation}
		Using the Carleman estimate \eqref{Carlem1} for the first forward equation of \eqref{forbacksyst}, and taking  a large enough \( s_1>0\) and \( \lambda_1>0\), we obtain for any \( s\geq s_1\) and \( \lambda\geq \lambda_1\)
		\begin{align*}
			\begin{aligned}
				&\, s\lambda^2 \mathbb{E} \int_Q \theta^2 \varphi \xi^2 \, dx \, dt + \lambda \mathbb{E} \int_Q \theta^2 |\nabla \xi|^2 \, dx \, dt \\
				&+ s\lambda^2 \mathbb{E} \int_\Sigma \theta^2 \varphi \xi_\Gamma^2 \,  d\sigma \, dt + \lambda \mathbb{E} \int_\Sigma \theta^2 |\nabla_\Gamma \xi_\Gamma|^2 \,  d\sigma \, dt \\
				&\leq C \mathbb{E} \left[ s\lambda \varphi(T) \theta^2(T) |\xi(T)|^2_{L^2(G)} + s\lambda \varphi(T) \theta^2(T) |\xi_\Gamma(T)|^2_{L^2(\Gamma)} \right] \\
				&\quad + C \mathbb{E} \int_Q \theta^2 \left| \frac{1}{\beta_1} \chi_{G_1}(x) p^1 + \frac{1}{\beta_2} \chi_{G_2}(x) p^2 - \eta_t y \right|^2 \, dx \, dt \\
				&\quad + C \mathbb{E} \int_\Sigma \theta^2 |\eta_t y_\Gamma|^2 \,  d\sigma \, dt.
			\end{aligned}
		\end{align*}
		It follows that
		\begin{align}\label{estimm142}
			\begin{aligned}
				& s\lambda^2 \mathbb{E} \int_Q \theta^2 \varphi \eta^2 y^2 \, dx \, dt + s\lambda^2 \mathbb{E} \int_\Sigma \theta^2 \varphi \eta^2 y_\Gamma^2 \,  d\sigma \, dt \\
				&\leq C \mathbb{E} \left[ s\lambda \varphi(T) \theta^2(T) |y(T)|^2_{L^2(G)} + s\lambda \varphi(T) \theta^2(T) |y_\Gamma(T)|^2_{L^2(\Gamma)} \right] \\
				&\quad + C \mathbb{E} \int_Q \theta^2 \eta^2 (|z^1|^2 + |z^2|^2) \, dx \, dt \\
				&\quad + C \mathbb{E} \int_Q \theta^2 \eta_t^2 y^2 \, dx \, dt + C \mathbb{E} \int_\Sigma \theta^2 \eta_t^2 y_\Gamma^2 \,  d\sigma \, dt.
			\end{aligned}
		\end{align}
		Applying the Carleman estimate \eqref{carl2} to the second backward equations of \eqref{forbacksyst}, and choosing a large \(s_2\), we get for all $s\geq s_2$
		\begin{align*}
			\begin{aligned}
				&\,s\lambda^2\mathbb{E}\int_Q \theta^2\varphi \left(|p^1|^2+|p^2|^2\right) \,\, dx\,dt + \lambda\mathbb{E}\int_Q \theta^2\left(|\nabla p^1|^2+|\nabla p^2|^2\right)\,\, dx\,dt\\
				&+s\lambda\mathbb{E}\int_Q\theta^2\varphi \left(|P^1|^2+|P^2|^2\right)\,\, dx\,dt+s\lambda^2\mathbb{E}\int_\Sigma \theta^2\varphi \left(|p_\Gamma^1|^2+|p_\Gamma^2|^2\right) \, d\sigma\,dt \\&+ \lambda\mathbb{E}\int_\Sigma \theta^2\left(|\nabla_\Gamma p_\Gamma^1|^2+|\nabla_\Gamma p_\Gamma^2|^2\right)\, d\sigma\,dt+s\lambda\mathbb{E}\int_\Sigma\theta^2\varphi \left(|\widehat{P}^1|^2+|\widehat{P}^2|^2\right)\, d\sigma\,dt\\
				&\leq C\mathbb{E}\int_Q\theta^2 (|\alpha_1\eta(y-y_{1,d})\chi_{G_{1,d}}-\eta_t z^1|^2+|\alpha_2\eta\left(y-y_{2,d})\chi_{G_{2,d}}-\eta_t z^2|^2\right) \,\, dx\,dt\\
				&\quad+C\mathbb{E}\int_\Sigma\theta^2 \left(|\eta_t z^1_\Gamma|^2+|\eta_t z^2_\Gamma|^2\right) \, d\sigma\,dt.
			\end{aligned}
		\end{align*}
		Then, we have that
		\begin{align}\label{estimat243}
			\begin{aligned}
				&\,s\lambda^2\mathbb{E}\int_Q \theta^2\varphi\eta^2\left(|z^1|^2+|z^2|^2\right) \,\, dx\,dt +s\lambda^2\mathbb{E}\int_\Sigma \theta^2\varphi\eta^2\left(|z_\Gamma^1|^2+|z_\Gamma^2|^2\right) \, d\sigma\,dt\\
				&\leq C\mathbb{E}\int_Q\theta^2 \eta^2 y^2\,\, dx\,dt+C\mathbb{E}\int_Q\theta^2 \eta^2 |y_{1,d}|^2\chi_{G_{1,d}} \,\,\, dx\,dt+C\mathbb{E}\int_Q\theta^2 \eta^2 |y_{2,d}|^2\chi_{G_{2,d}} \,\,\, dx\,dt\\
				&\quad+C\mathbb{E}\int_Q\theta^2 \eta_t^2 \left(|z^1|^2+|z^2|^2\right)\,\, dx\,dt+C\mathbb{E}\int_\Sigma\theta^2 \eta_t^2 \left(|z^1_\Gamma|^2+|z^2_\Gamma|^2\right) \, d\sigma\,dt.
			\end{aligned}
		\end{align}
		Adding \eqref{estimm142} and \eqref{estimat243} and taking a large $s\geq s_3=\max(s_1,s_2)$, we obtain that
		\begin{align*}
			\begin{aligned}
				&\,s\lambda^2\mathbb{E}\int_Q \theta^2\varphi \eta^2 \left(y^2+|z^1|^2+|z^2|^2\right) \,\, dx\,dt+s\lambda^2\mathbb{E}\int_\Sigma \theta^2\varphi \eta^2 y_\Gamma^2 \, d\sigma\,dt\\
				&\leq C\mathbb{E}\left[s\lambda\varphi(T)\theta^2(T) |y(T)|^2_{L^2(G)}+s\lambda\varphi(T)\theta^2(T)|y_\Gamma(T)|^2_{L^2(\Gamma)}\right]\\
				&\quad\,+C\theta^2(T)\mathbb{E}\int_Q \left(|y_{1,d}|^2\chi_{G_{1,d}}+|y_{2,d}|^2\chi_{G_{2,d}}\right)\,\, dx\,dt+C\mathbb{E}\int_Q\theta^2 \eta_t^2 \left(y^2+|z^1|^2+|z^2|^2\right)\,\, dx\,dt\\
				&\quad+C\mathbb{E}\int_\Sigma\theta^2 \eta_t^2 \left(y_\Gamma^2+|z^1_\Gamma|^2+|z^2_\Gamma|^2\right) \, d\sigma\,dt.
			\end{aligned}
		\end{align*}
		It follows that
		\begin{align}\label{ineqq44}
			\begin{aligned}
				&\,s\lambda^2\mathbb{E}\int_Q \theta^2\varphi \eta^2 \left(y^2+|z^1|^2+|z^2|^2\right)  \,\, dx\,dt+s\lambda^2\mathbb{E}\int_\Sigma \theta^2\varphi \eta^2 y_\Gamma^2 \, d\sigma\,dt\\
				&\leq C\varphi(T)\theta^2(T)\left[\mathbb{E}\int_G s\lambda y^2(T) \, dx+\mathbb{E}\int_\Gamma s\lambda y^2_\Gamma(T) \,  d\sigma+M_1\right]\\
				&\quad\,+C\mathbb{E}\int_Q\theta^2 \eta_t^2 \left(y^2+|z^1|^2+|z^2|^2\right)\,\, dx\,dt+C\mathbb{E}\int_\Sigma\theta^2 \eta_t^2 \left(y_\Gamma^2+|z^1_\Gamma|^2+|z^2_\Gamma|^2\right) \, d\sigma\,dt,
			\end{aligned}
		\end{align}
		where $M_1=\mathbb{E}\displaystyle\int_Q \left(|y_{1,d}|^2\chi_{G_{1,d}}+|y_{2,d}|^2\chi_{G_{2,d}}\right)\,\, dx\,dt$. \\
		On the other hand, we have that
		\begin{align*}
			\begin{aligned}
				&\,\mathbb{E}\int_Q\theta^2 \eta_t^2 \left(y^2+|z^1|^2+|z^2|^2\right)\,\, dx\,dt+\mathbb{E}\int_\Sigma\theta^2 \eta_t^2 \left(y_\Gamma^2+|z^1_\Gamma|^2+|z^2_\Gamma|^2\right) \, d\sigma\,dt\\
				&\leq C\mathbb{E}\int_{t_2}^{t_1}\int_G\theta^2  \left(y^2+|z^1|^2+|z^2|^2\right)\,\, dx\,dt+\mathbb{E}\int_{t_2}^{t_1}\int_\Gamma\theta^2  \left(y_\Gamma^2+|z^1_\Gamma|^2+|z^2_\Gamma|^2\right) \, d\sigma\,dt,
			\end{aligned}
		\end{align*}
		which implies that
		\begin{align}\label{ine445}
			\begin{aligned}
				&\,\mathbb{E}\int_Q\theta^2 \eta_t^2 \left(y^2+|z^1|^2+|z^2|^2\right)\,\, dx\,dt+\mathbb{E}\int_\Sigma\theta^2 \eta_t^2 \left(y_\Gamma^2+|z^1_\Gamma|^2+|z^2_\Gamma|^2\right) \, d\sigma\,dt\\
				&\leq C\theta^2(t_1)\left[\mathbb{E}\int_Q   \left(y^2+|z^1|^2+|z^2|^2\right)\,\, dx\,dt+\mathbb{E}\int_\Sigma   \left(y_\Gamma^2+|z^1_\Gamma|^2+|z^2_\Gamma|^2\right) \, d\sigma\,dt\right].
			\end{aligned}
		\end{align}
		From \eqref{ineqq44} and \eqref{ine445}, it is easy to see that
		\begin{align}\label{ineqq4477}
			\begin{aligned}
				&\,s\lambda^2\mathbb{E}\int_{t_0}^T\int_G  \left(y^2+|z^1|^2+|z^2|^2\right)  \,\, dx\,dt+s\lambda^2\mathbb{E}\int_{t_0}^T\int_\Gamma  y_\Gamma^2 \, d\sigma\,dt\\
				&\leq C\varphi(T)\theta^2(T)\left[\mathbb{E}\int_G s\lambda y^2(T) \, dx+\mathbb{E}\int_\Gamma s\lambda y^2_\Gamma(T) \,  d\sigma+M_1\right]\\
				&\quad+C\theta^{-2}(t_0)\theta^2(t_1)\left[\mathbb{E}\int_Q   \left(y^2+|z^1|^2+|z^2|^2\right)\,\, dx\,dt+\mathbb{E}\int_\Sigma   \left(y_\Gamma^2+|z^1_\Gamma|^2+|z^2_\Gamma|^2\right) \, d\sigma\,dt\right].
			\end{aligned}
		\end{align}
		By Itô's formula, we obtain 
		\begin{align*}
			\begin{aligned}
				&-\mathbb{E}\int_{t_0}^T\int_G dy^2\,\, dx-\mathbb{E}\int_{t_0}^T\int_\Gamma dy_\Gamma^2\, d\sigma\\
				&\leq C\left[\mathbb{E}\int_{t_0}^T\int_G  \left(y^2+|z^1|^2+|z^2|^2\right)  \,\, dx\,dt+\mathbb{E}\int_{t_0}^T\int_\Gamma  y_\Gamma^2 \, d\sigma\,dt\right]. 
			\end{aligned}
		\end{align*}
		Hence,
		\begin{align}\label{esmtt466}
			\begin{aligned}
				\mathbb{E}\int_G y^2(t_0,x) \, dx+\mathbb{E}\int_\Gamma y_\Gamma^2(t_0,x) \,  d\sigma&\leq \mathbb{E}\int_G y^2(T,x) \, dx+\mathbb{E}\int_\Gamma y_\Gamma^2(T,x) \,  d\sigma\\
				&\quad+C\left[\mathbb{E}\int_{t_0}^T\int_G  \left(y^2+|z^1|^2+|z^2|^2\right)  \,\, dx\,dt+\mathbb{E}\int_{t_0}^T\int_\Gamma  y_\Gamma^2 \, d\sigma\,dt\right].
			\end{aligned}
		\end{align}
		Combining \eqref{ineqq4477} and \eqref{esmtt466} and fixing the parameter $\lambda=\lambda_1$, we get that
		\begin{align*}
			\begin{aligned}
				&\,\mathbb{E}\int_G y^2(t_0,x) \, dx+\mathbb{E}\int_\Gamma y_\Gamma^2(t_0,x) \,  d\sigma\\
				&\leq C\theta^2(T)\left[\mathbb{E}\int_G y^2(T,x) \, dx+\mathbb{E}\int_\Gamma y^2_\Gamma(T,x) \,  d\sigma+M_1\right]\\
				&\quad+C\theta^{-2}(t_0)\theta^2(t_1)\left[\mathbb{E}\int_Q   \left(y^2+|z^1|^2+|z^2|^2\right)\,\, dx\,dt+\mathbb{E}\int_\Sigma   \left(y_\Gamma^2+|z^1_\Gamma|^2+|z^2_\Gamma|^2\right) \, d\sigma\,dt\right].
			\end{aligned}
		\end{align*}
		Then, we deduce that
		\begin{align}\label{estmmf}
			\begin{aligned}
				&\,\mathbb{E}\int_G y^2(t_0,x) \, dx+\mathbb{E}\int_\Gamma y_\Gamma^2(t_0,x) \,  d\sigma\\
				&\leq Ce^{2se^{\lambda_1 T}}\left[\mathbb{E}\int_G y^2(T,x) \, dx+\mathbb{E}\int_\Gamma y^2_\Gamma(T,x) \,  d\sigma+M_1\right]\\
				&\quad+Ce^{-2s(e^{\lambda_1 t_0}-e^{\lambda_1 t_1})}\left[\mathbb{E}\int_Q   \left(y^2+|z^1|^2+|z^2|^2\right)\,\, dx\,dt+\mathbb{E}\int_\Sigma   \left(y_\Gamma^2+|z^1_\Gamma|^2+|z^2_\Gamma|^2\right) \, d\sigma\,dt\right],
			\end{aligned}
		\end{align}
		for all $s\geq s_3$. Replacing $C$ by $Ce^{Cs_3}$, then \eqref{estmmf} holds for all $s>0$. Choosing $s\geq0$ minimizing the right-hand side of \eqref{estmmf}, we end up with
		\begin{align}
			\begin{aligned}
				&\,\mathbb{E}\int_G y^2(t_0,x)\, dx+  \mathbb{E}\int_\Gamma y_\Gamma^2(t_0,x)\,  d\sigma \\
				&\leq C\left(\mathbb{E}\int_Q y^2 \,\, dx\,dt+\mathbb{E}\int_\Sigma   y_\Gamma^2 \, d\sigma\,dt+M_0\right)^{1-\kappa}\left(\mathbb{E}\int_G y^2(T,x)\, dx+  \mathbb{E}\int_\Gamma y_\Gamma^2(T,x)\,  d\sigma +M_1\right)^{\kappa},
			\end{aligned}
		\end{align}
		where 
		$$M_0=\mathbb{E}\int_Q \left(|z^1|^2+|z^2|^2\right)\,\, dx\,dt+\mathbb{E}\int_\Sigma   \left(|z^1_\Gamma|^2+|z^2_\Gamma|^2\right) \, d\sigma\,dt,$$
		and
		$$\kappa=\frac{2(e^{\lambda_1 t_0}-e^{\lambda_1 t_1})}{C+2(e^{\lambda_1 t_0}-e^{\lambda_1 t_1})}.$$
		This concludes the proof of Theorem \ref{thmm1.1}.
	\end{proof}
	
	\section{Conclusion}
	In this work, we study a multi-objective inverse initial problem with a Nash strategy for stochastic reaction-diffusion equations with dynamic boundary conditions. By first characterizing the Nash equilibrium for some given functionals \(J_1\) and \(J_2\), we reduce the multi-objective inverse initial problem to a classical inverse initial problem related to the coupled forward-backward stochastic system \eqref{eqq4.7}. We then establish the interpolation inequality \eqref{interine2.3}, from which we derive two main results: backward uniqueness and a conditional stability estimate for the solutions of \eqref{eqq1.1}. The proofs are based on new Carleman estimates for both forward and backward stochastic parabolic equations with dynamic boundary conditions, employing a weighted identity method with the simple weight function \(\varphi = e^{\lambda t}\) (where \(\varphi\) is \(x\)-independent), and \(\lambda > 0\) is a suitably large parameter.

\end{document}